\documentclass[a4paper,twoside,10pt]{article}

\usepackage[english]{babel}
\usepackage[latin1]{inputenc}
\usepackage{lmodern,color}
\usepackage[T1]{fontenc}
\usepackage{indentfirst}
\usepackage{amsmath,amssymb}

\usepackage[a4paper,top=3cm,bottom=3cm,left=3cm,right=3cm,%
bindingoffset=5mm]{geometry}
\usepackage{lmodern}
\usepackage[T1]{fontenc}
\usepackage{lettrine}
\usepackage{booktabs}

\usepackage{authblk}
\usepackage{emp}
\usepackage{amsthm}
\usepackage{amsmath,amssymb,amsthm}
\usepackage{braket}
\usepackage{chngpage}
\usepackage{cases}
\usepackage{graphicx}
\usepackage{epstopdf}
\usepackage{tabularx}
\usepackage{multirow}

\newcommand{\numberset}{\mathbb}
\newcommand{\N}{\numberset{N}}
\newcommand{\R}{\numberset{R}}

\newcommand{\Pk}{\numberset{P}}
\renewcommand{\epsilon}{\varepsilon}
\renewcommand{\theta}{\vartheta}
\renewcommand{\rho}{\varrho}
\renewcommand{\phi}{\varphi}

\def\P0{{\Pi^{0, E}_{k-2}}}

\def\PN{{\Pi^{\nabla, E}_k}}
\def\PP0{{\boldsymbol{\Pi}^{0, E}_{k-1}}}

{\left\lbrace\begin{array}{@{}l@{}}}%
{\end{array}\right.}
\theoremstyle{definition}

\theoremstyle{remark}
\newtheorem{remark}{Remark}[section]

\theoremstyle{remark}
\newtheorem{test}{Test}[section]

\theoremstyle{plain}
\newtheorem{theorem}{Theorem}[section]
\newtheorem{proposition}{Proposition}[section]

\newtheorem{lemma}{Lemma}[section]

\newcommand{\blu}{\color{blue}}
\newcommand{\red}{\color{red}}
\usepackage{lipsum}

\usepackage{authblk}
\usepackage[percent]{overpic}
\usepackage{setspace}

\author[1]{L. Beir\~ao da Veiga \thanks{lourenco.beirao@unimib.it}}

\author[1]{A. Russo \thanks{alessandro.russo@unimib.it}}

\author[1]{G. Vacca \thanks{giuseppe.vacca@unimib.it}}

\affil[1]{Dipartimento di Matematica e Applicazioni,  Universit\`a degli Studi di Milano Bicocca, Via Roberto Cozzi 55 - 20125 Milano, Italy}  

\title{\textbf{The Virtual Element Method with curved edges}}

\date{\today}

\begin{document}

\maketitle
\begin{abstract}
In this paper we initiate the investigation of Virtual
Elements with curved faces. We consider the case of a fixed curved
boundary in two dimensions, as it happens in the approximation of
problems posed on a curved domain or with a curved interface. While an approximation of the
domain with polygons leads, for degree of accuracy $k \geq 2$, to a sub-optimal rate of
convergence, we show (both theoretically and numerically) that the
proposed curved VEM lead to an optimal rate of convergence. 
\end{abstract}

The copyright of the original paper is owned by EDP Sciences and SMAI. The original publication appears on the journal M2AN (Mathematical Modelling and Numerical Analysis) and can be found at 
\texttt{www.esaim-m2an.org}.

\section{Introduction}
\label{sec:1}

The Virtual Element Method (VEM) was introduced in \cite{volley, VEM-hitchhikers} as a generalization of the Finite Element Method that allows for general polygonal and polyhedral meshes. Polytopal meshes can be very useful for a wide range of reasons, including meshing of the domain (such as cracks) and data (such as inclusions) features, automatic use of hanging nodes, moving meshes, adaptivity. By avoiding the explicit construction of the local basis functions, Virtual Elements can easily handle general polygons/polyhedrons without the need of an overly complex construction. Since its first appearence, VEM has shared a very good success both in the mathematical and engineering literature: a very small sample of works is given by 
\cite{
GTP14,
Helmholtz-VEM,
Berrone-VEM,
hourglass,
vacca2016virtual, 
wriggers,
Stokes:divfree, 
dassi,
Steklov-VEM,
gardini,
mascotto,
chernov}.

The scope of the present paper is to develop a first study (on a simple model elliptic problem in 2D) of Virtual Elements with curved faces. Indeed, all the VEM papers in the literature make use of polygonal and polyhedral meshes, i.e. with straight edges and faces. On the other hand, as recognized in the finite element (FEM) literature, especially for high order methods the approximation of the domain by facets introduces an error that can dominate the analysis. This has led, in the FEM literature, to the development of non affine isoparametric elements, that is elements that are mapped with higher order polynomial maps that allow for a better approximation of the domain of interest \cite{scott:1975, Zlamal:1973, Ciarlet:1972, Lenoir:1986}. 
Another related technology is that of Isogeometric Analysis 
\cite{Hughes:2005, Hughes:2006, Hughes:2009}
that proposes an exact approximation of CAD domains by making use (in building the discrete spaces) of the same NURBS maps that parametrize the geometry.

In the context of Virtual Elements, one can exploit the peculiar construction of the method that (1) does not need an explicit expression of the basis functions and (2) is directly defined in physical space, i.e. no reference element is used. This allows to define discrete spaces also on elements that are curved in such a way to exactly represent the domain of interest. The needed ingredient is a (piecewise regular) parametrization of the boundary of the domain. 
The advantage with respect to isoparametric Finite Elements is that no approximation of the domain (even by polynomial functions) is introduced and, most importantly, no ``wise positioning'' of the isoparametric nodes is needed 
(see \cite{Lenoir:1986}). 
The advantage with respect to Isogeometric Analysis is that only the boundary parametrization is needed (and not the full volume parametrization), which is much more readily available in many CAD applications. Clearly this comes at a price, that is the absence of a reference element (or parametric domain) that can make the construction more costly from the computational standpoint (indeed we are not claiming that our method is superior, but only that it has appealing characteristics when compared to isoparametric FEM and IGA). 

We must also point out that there are other effective approaches for the accurate treatment of curved domains in the finite element framework (an interesting survey can be found in \cite{Sevilla:2011}).
Finally, to the best of our knowledge, the only numerical methods, in addition to the one here presented, that can handle curved polytopal meshes are \cite{Brezzi:2006, dipietro}.
The paper is organized as follows. In Section \ref{sec:2} we present the construction of the method on a model elliptic problem. In Section \ref{sec:3} we present the theoretical convergence analysis of the method, including interpolation estimates for the curved VEM spaces and stability bounds for the associated discrete bilinear form. 
In Section \ref{sec:4} we investigate the numerical integration on curved polygons.
In Section \ref{sec:tests} we present a set of numerical tests, also comparing the result with the original ``straight edge'' case.   
It is interesting to note that, although the initial construction and the theoretical proofs are clearly more involved, at the practical level the coding of the method with curved edges turns out to be essentially the same as in the ``straight edge'' case (the only difference being in the integration on edges and elements). Somehow, the present case fits very naturally into the Virtual Element setting.

\section{Virtual element space on curved domains}
\label{sec:2}

\subsection{Assumptions on the curved domains and notation}
\label{sub:2.1}

We start by reviewing the mathematical basis of our problem. Let $\Omega$ be a bounded open subset
of $\R^2$ whose boundary $\Gamma := \partial \Omega$  is made up of a finite number of smooth curves $\{\Gamma_i\}_{i=1, \dots, N}$  (see Figure \ref{fig:domain}), and we assume that
an elliptic boundary value problem is given over $\Omega$. 
\begin{figure}[!h]
\center{
\includegraphics[scale=0.50]{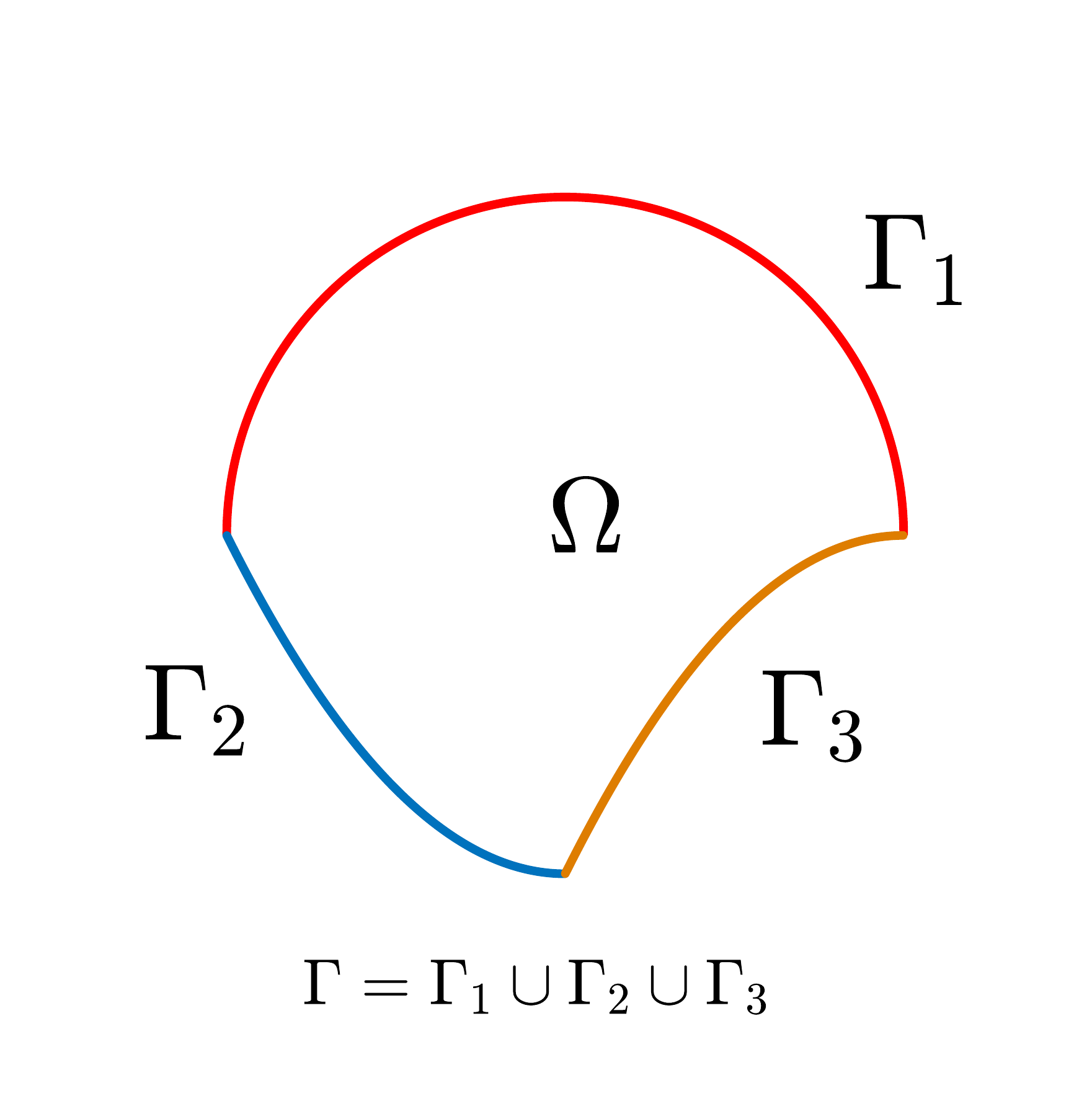} 
\vspace*{-0.5cm}
\caption{Example of curved domain.}
\label{fig:domain}
}
\end{figure}

We assume that: 
\begin{description}
\item [$\mathbf{(A0)}$] the boundary $\Gamma$ is Lipschitz and each curve $\Gamma_i$ of $\Gamma$ is sufficiently smooth, for instance we require that $\Gamma_i$ is of class $C^{m+1}$ (i.e. that $\Gamma$ belongs to $C^{m+1}$  piecewise) with $m \geq 0$.
\end{description}
Let $\gamma_i \colon I_i \to \Gamma_i$ be a given regular and invertible $C^{m+1}$-parametrization of the curve $\Gamma_i$ for $i=1, \dots, N$, where  $I_i := [a_i, b_i] \subset \R$ is a closed interval. Let $\ell_{\Gamma_i}$ be the length of the curve $\Gamma_i$, then we denote with $\chi_i \colon [0, \ell_{\Gamma_i}] \to \Gamma_i$  the arc-length, i.e. the unit-speed parametrization of $\Gamma_i$. Finally we denote with $\zeta_i \colon I_i \to [0, \ell_{\Gamma_i}]$ the natural parameter of $\gamma_i$, i.e.  the map 
\[
\zeta_i(t) := \int_{a_i}^t \|\gamma_i'(s)\| \, {\rm d} s \qquad \text{for all $t \in I_i$.} 
\]
It is clear that $\gamma_i(t) = \chi_i(\zeta_i(t))$ for all $t \in I_i$ (see also Figure \ref{fig:global}).
Moreover being $\|\gamma_i'\|_{L^{\infty}} > 0$,  both $\zeta_i$ and $\zeta_i^{-1}$ are in $W^{m+1, \infty}$.
\\
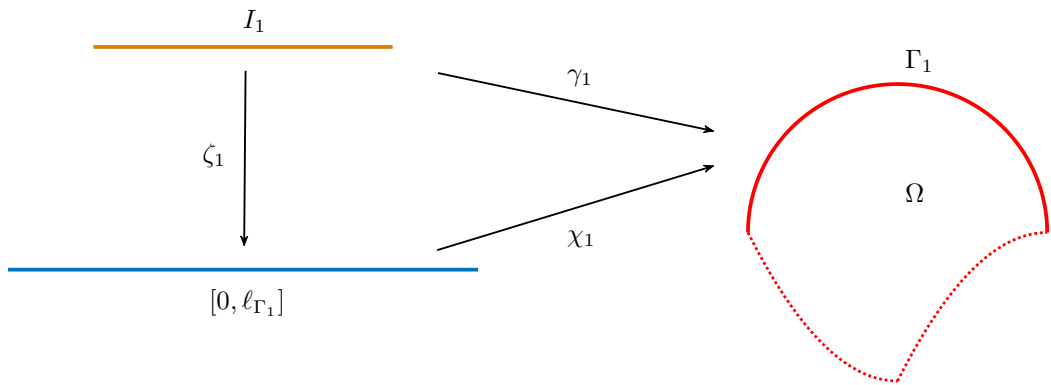
\begin{figure}[!h]
\center
{
\hspace*{-2cm}%
\begin{overpic}[scale=0.35]{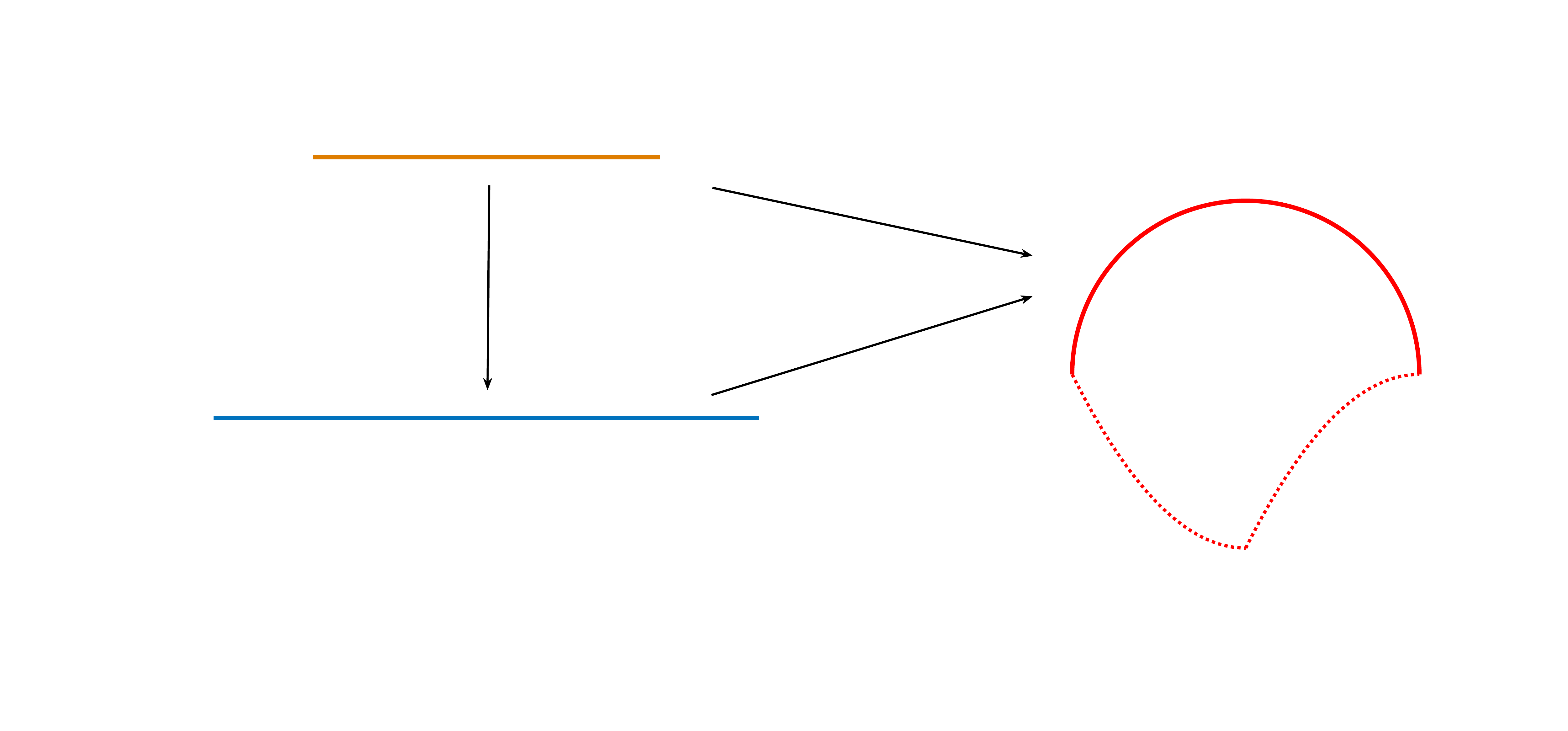} 
\put (55,22) {$\chi_1$}
\put (55,34) {$\gamma_1$}
\put (28,28) {$\zeta_1$}
\put (28.5,17) {$[0, \ell_{\Gamma_1}]$}
\put (31,38) {$I_1$}
\put (80,35) {$\Gamma_1$}
\put (80,25) {$\Omega$}
\end{overpic}
\vspace*{-2cm}
\caption{Parametrizations of the curve $\Gamma_1$ by means of the given parametrization $\gamma_1$ and the arc-length parametrization $\chi_1$, and natural parameter $\zeta_1$.}
\label{fig:global}
}
\end{figure}
Since all the parts $\Gamma_i$ of $\partial \Omega$ will be treated in the same way, in the following we will drop the index $i$ from all the involved maps and parameters (in order to obtain a lighter notation).

We consider as our model problem the elliptic equation:
\begin{equation}
\label{eq:continua}
\left\{
\begin{aligned}
& \text{find $u \in V := H_0^1(\Omega)$, such that} \\
&  a(u, \, v) = (f, \, v) \qquad & \text{for all $v \in V$,} \\
\end{aligned}
\right.
\end{equation}
where, using standard notation, $(\cdot, \, \cdot)$ denotes the $L^2$-inner product on $\Omega$ and $a(\cdot, \, \cdot) = (\nabla \cdot, \, \nabla \cdot)$.
It is well known that Problem \eqref{eq:continua} has a unique solution $u$.
%

Using the tool of the present paper and on the light of the existing VEM literature, it is possible to extend  this approach to other scalar elliptic equations (such as diffusion-convection-reaction with variable coefficients).

We will indicate the classical Sobolev semi-norms (and analogously for the norms) with the shorter symbols 
\[
|v|_{s}=|v|_{H^s(\Omega)} 
\quad \text{for all $v \in H^s(\Omega)$} 
\qquad \text{and} \qquad 
|v|_{s,\omega}=|v|_{H^s(\omega)}
\quad \text{for all $v \in H^s(\omega)$} 
\]
for any non-negative real number $s$ and generic open measurable set $\omega$.
We refer to \cite{lions-magenes} for the definition of Sobolev norms with real index.



\subsection{Virtual Element Spaces}
\label{sub:2.2}

Let $\set{\Omega_h}_h$ be a sequence of decompositions of $\Omega$ into general polygons $E$ completed along $\Gamma$  by elements whose boundary contains an arc $\subset \Gamma$, and let
\[
h_E := {\rm diameter}(E) , \quad
h := \sup_{E \in \Omega_h} h_E .
\]
We suppose that for all $h$, each element $E$ in $\Omega_h$ fulfils the following assumptions:
\begin{description}
\item [$\mathbf{(A1)}$] $E$ is star-shaped with respect to a ball $B_E$ of radius $ \geq\, \rho \, h_E$, 
\item [$\mathbf{(A2)}$] the length of any (possibly curved) edge of $E$ is $\geq \rho \, h_E$, 
\end{description}
where $\rho$ is a positive constant (see Figure \ref{fig:mesh} for an example of such decomposition). We remark that the hypotheses above, though not too restrictive in many practical cases, 
could be possibly further relaxed, as investigated in \cite{2016stability} for the case with straight edges. 
\begin{figure}[!h]
\center{
\hspace*{-4cm}%
\includegraphics[scale=0.45]{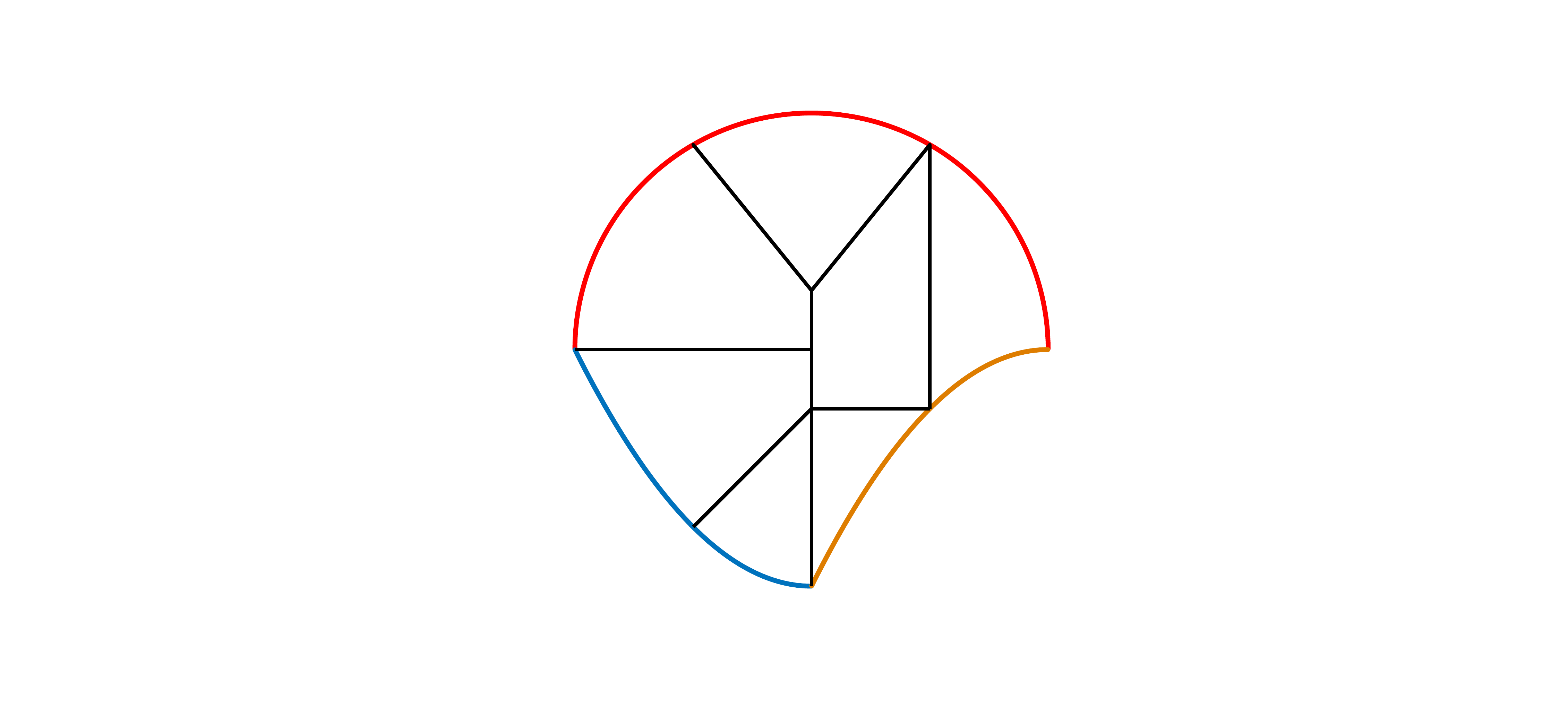} 
\vspace*{-1.7cm}
\caption{Mesh decomposition of the domain $\Omega$.}
\label{fig:mesh}
}
\end{figure}

In the following $C$ will denote a generic positive constant independent of the mesh diameter $h$ (possibly dependent on $\rho$) and that may change at each occurrence, 
and $\lesssim$ will denote a bound up to  $C$. 

For each element $E \in \Omega_h$ having an edge $e$ lying on the boundary $\Gamma$, with a slight abuse we use the same notations both for the global parametrizations of $\Gamma$ and the local parametrizations of $e$.
We denote with 
$\gamma \colon I_e \subset I \to e$  the restriction of $\gamma \colon I \to \Gamma$ having image $e$. 
Similarly, defined $\ell_e$ the length of $e$, we denote with $\chi \colon [0, \, \ell_e] \to e$ the arc-length parametrization of the curved edge $e$
 and with $\zeta \colon I_e \to [0, \, \ell_e]$ the associated natural parameter  (see Figure \ref{fig:globallocal} and Figure \ref{fig:local}).

\begin{figure}[!h]
\center
{
\hspace*{-2cm}%
\begin{overpic}[scale=0.35]{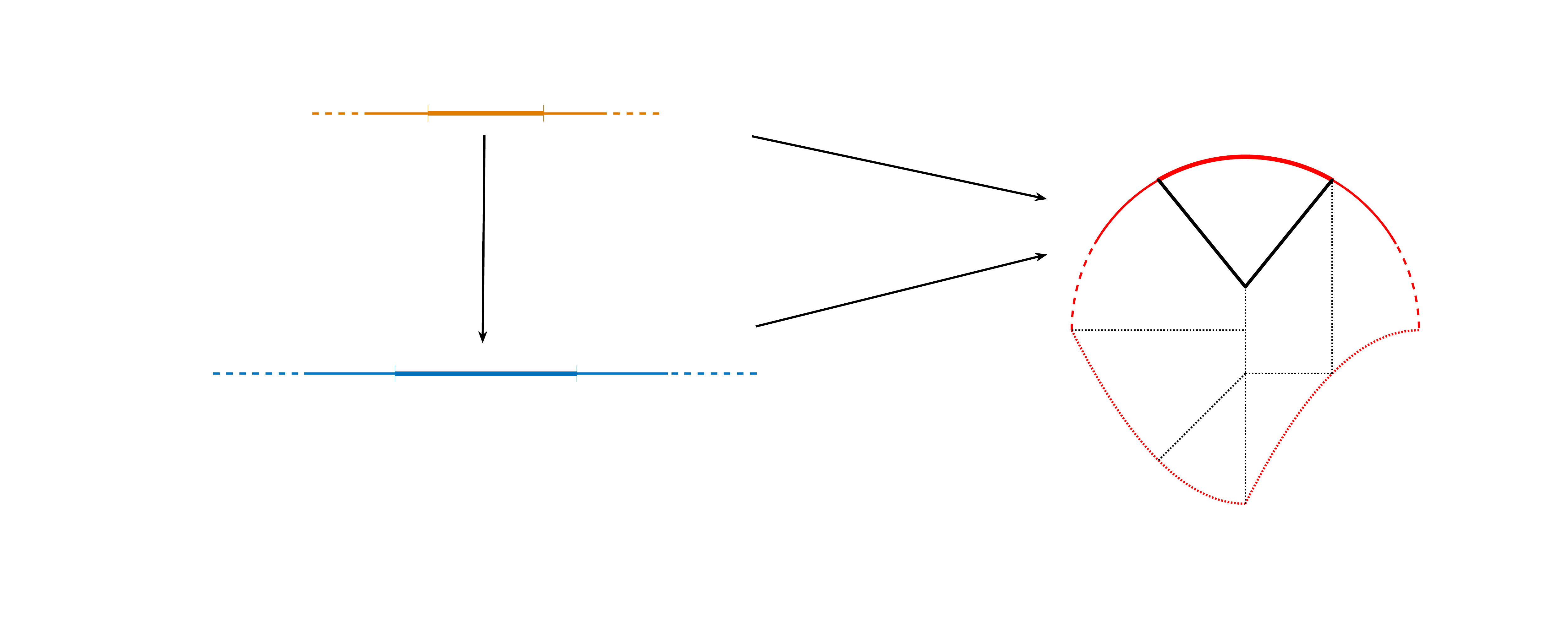} 
\put (56,20) {$\chi$}
\put (56,31) {$\gamma$}
\put (28,26) {$\zeta$}
\put (12,14) {$[0, \ell_{\Gamma}]$}
\put (27,14) {$[c, \, c + \ell_{e}]$}
\put (19,35) {$I$}
\put (31,35) {$I_e$}
\put (88,29) {$\Gamma$}
\put (80,21) {$\Omega$}
\put (80,26) {$E$}
\put (80,31.5) {$e$}
\end{overpic}
\vspace*{-2cm}
\caption{Global parametrizations of the curve $\Gamma$ and the associated restriction to $e$.}
\label{fig:globallocal}
}
\end{figure}
\begin{figure}[!h]
\center
{
\hspace*{-3cm}%
\begin{overpic}[scale=0.35]{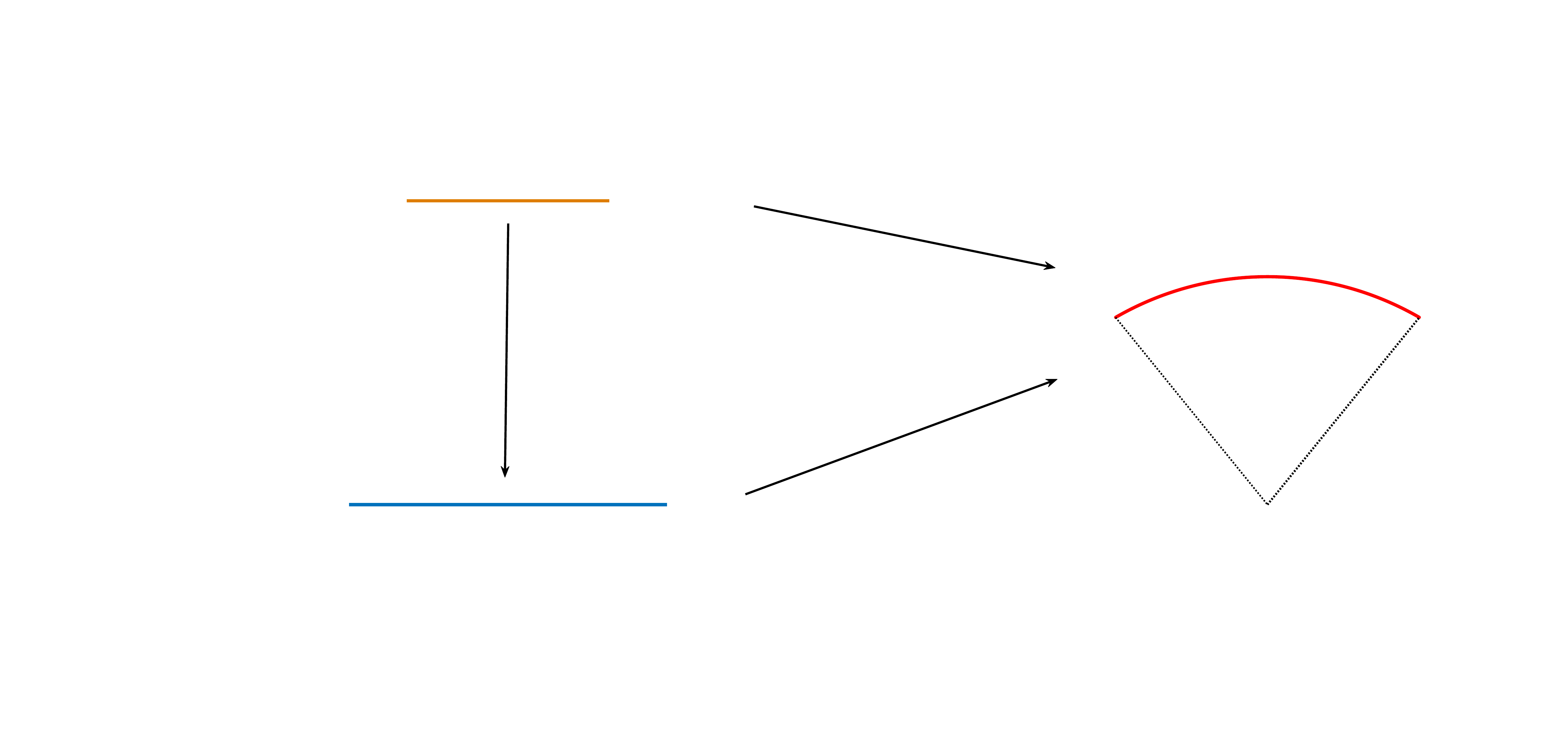} 
\put (57,16) {$\chi$}
\put (57,33) {$\gamma$}
\put (29,214) {$\zeta$}
\put (30,12) {$[0, \ell_{e}]$}
\put (32,35) {$I_e$}
\put (80,30) {$e$}
\put (80,23) {$E$}
\end{overpic}
\vspace*{-2cm}
\caption{Local parametrizations of the curved edge $e$ by means of the given parametrization $\gamma$ and the arc-length parametrization $\chi$, and natural parameter $\zeta$.}
\label{fig:local}
}
\end{figure}
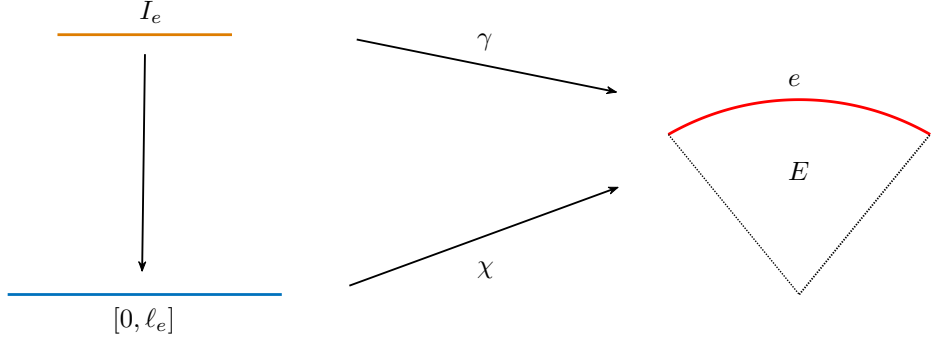
We remark that, since the parametrization $\gamma \colon I \to \Gamma$ is fixed once and for all, under the assumption $\mathbf{(A0)}$, it easily follows that for any curved edge $e \subseteq \partial E$, the length of the interval $I_e$ is comparable with the diameter $h_E$ of
the element $E$. 
Indeed, since $\ell_e = \int_{I_e} \|\gamma'(s)\| \, {\rm d}s$ and $\gamma$, $\gamma^{-1} \in W^{1, \infty}$ are fixed once and for all, it holds that 
the length of the interval $I_e$ is comparable with  $\ell_e$.
Moreover, since $\gamma$ is fixed, when $h$ approaches  zero, roughly speaking,  the straight segment  $e'$ that shares the vertexes with  $e$ approaches the curved edge $e$. 
Therefore by assumption  $\mathbf{(A2)}$, for sufficiently small $h$, the length $\ell_e$ of the curved edge $e$ is comparable with the diameter $h_E$.
Since we have assumed that both $\gamma$  and $\chi$ are  invertible mappings on their respective domains of definition, following the scheme of Figure \ref{fig:local}, we can establish the following correspondences:
\begin{equation}
\label{eq:correspondences}
\begin{aligned}
& \hat{v} = v \circ \gamma &\qquad \check{v} &= v \circ \chi & \qquad &\text{for all \, $v \colon e \to \R$,} \\
& v = \hat{v} \circ \gamma^{-1}  &\qquad  \check{v} &= \hat{v} \circ \zeta^{-1} &   \qquad &\text{for all \, $\hat{v} \colon I_e \to \R$,} \\
& v = \check{v} \circ \chi^{-1} &\qquad \hat{v} &= \check{v} \circ \zeta &   \qquad &\text{for all \, $\check{v} \colon [0, \, \ell_e] \to \R$.} \\
\end{aligned}
\end{equation}

%
%
%
By definition of Sobolev norms on curves, using the notation in \eqref{eq:correspondences}, we can set
\begin{equation}
\label{eq:normcurves}
\|v\|_{s, e} := \|\check{v}\|_{H^s((0,\, \ell_e))} , \qquad \text{and} \qquad |v|_{s, e} := |\check{v}|_{H^s((0,\, \ell_e))} \qquad \text{for all positive real numbers $s$.}
\end{equation}
%
%
%
%
%
Let $k \geq 1$ and $m \geq k$ (cf. assumption $\mathbf{(A0)}$).  Let $e \subset \Gamma$, then we introduce the following useful notation:
\begin{equation}
\label{eq:ptilde}
\widetilde{\Pk_k}(e) := \{  \widetilde{q_k} = \widehat{q_k} \circ \gamma^{-1} \quad \text{s.t.} \quad \widehat{q_k} \in \Pk_k(I_e)\}, 
\end{equation}
i.e. $\widetilde{\Pk_k}(e)$ is made of all functions that are polynomials with respect to the parametrization $\gamma$.
It is worth pointing out that the space $\widetilde{\Pk_k}(e)$ defined in \eqref{eq:ptilde} generally contains functions which are not the restriction to $e$ of polynomials living on $\Omega$; in particular it corresponds to the space of polynomials if $\gamma$ is an affine map.

From now on, for the sake of simplicity, we assume that every element $E \in \Omega_h$ has at most one edge lying on $\Gamma$ (therefore only one edge of $E$ will be curved and the rest will be straight).  
Treating the case with more curved edges follows trivially the same construction. Moreover we assume that every curved edge lies on only one regular curve, i.e $e \subseteq \Gamma_i$ (this second condition is instead mandatory for the approach followed in this paper).
Let $E \in \Omega_h$ with $\partial E = \cup_{i=1}^{N_E} e_i$, where $e_1 \subset \Gamma$, and $e_2, \cdots, e_{N_E}$ are straight segments, we now introduce the local virtual space on the curved element $E$ (see Figure \ref{fig:element}):
\begin{multline}
\label{eq:virtualpace}
V_h^E := \bigl\{ v \in H^1(E) \cap C^0(E) \quad \text{s.t} \quad  -\Delta v \in \Pk_{k-2}(E), \bigr. \\
\bigl.  v_{|e_1} \in \widetilde{\Pk_k}(e_1), \quad v_{|e_i} \in \Pk_k(e_i) \quad \text{for $i=2, \dots, N_E$}  \bigr\}.
\end{multline}

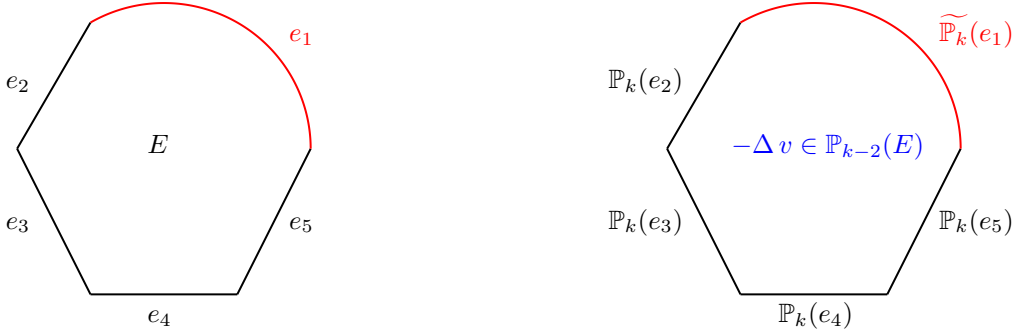
\begin{figure}[!h]
\center{
\hspace*{-16cm}%
\begin{overpic}[scale=0.25]{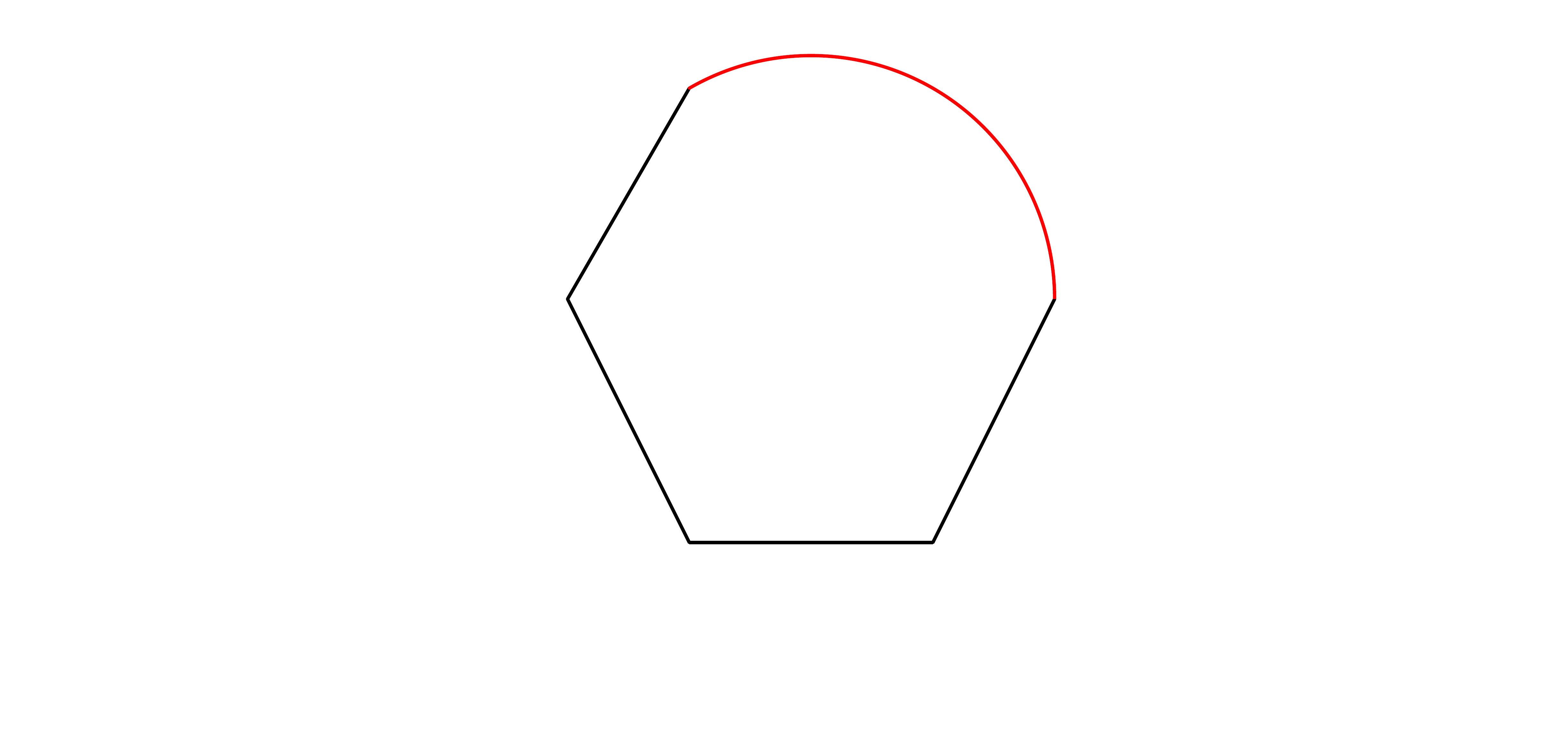}
\put (65,40) {{\red{$e_1$}}}
\put (35,35) {{$e_2$}}
\put (35,20) {{$e_3$}}
\put (50,10) {{$e_4$}}
\put (65,20) {{$e_5$}}
\put (50,28) {{$E$}}
\end{overpic}
\hspace*{-4cm}%
\begin{overpic}[scale=0.25]{elementpic.pdf}
\put (43,28) {{\blu{$- \Delta \,v  \in \Pk_{k-2}(E)$}}}
\put (65,40) {{\red{$\widetilde{\Pk_k}(e_1)$}}}
\put (30,35) {{$\Pk_k(e_2)$}}
\put (30,20) {{$\Pk_k(e_3)$}}
\put (48,10) {{$\Pk_k(e_4)$}}
\put (65,20) {{$\Pk_k(e_5)$}}
\end{overpic}
\hspace*{-15cm}%
\vspace*{-1cm}
\caption{Example of virtual space on curved element $E$.}
\label{fig:element}
}
\end{figure}
\begin{remark}
\label{rm:non inclusion}
From definition \eqref{eq:virtualpace} it is clear that differently to the standard case with straight edges, $\Pk_k(E) \not \subset V_h^E$.  Therefore we need to be more careful in the definition of the discrete approximated bilinear form $a_h(\cdot, \, \cdot)$ (see Section \ref{sub:2.3}) and in the interpolation analysis (see Section \ref{sec:3}). 
On the other hand it is easy to check that $\Pk_0(E) \subset V_h^E$.
\end{remark}

The corresponding degrees of freedom are chosen in accordance with the non curved case  (see Figure \ref{fig:dofsloc}).  

\begin{proposition}[Degrees of Freedom]
\label{prp:dofs}
The following linear operators $\boldsymbol{D}$, split into boundary operators $\boldsymbol{D^{\partial}}$ and internal ones $\boldsymbol{D^{o}}$,
constitute a set of degrees of freedom (DoFs) for $V_h^E$:
\begin{itemize}
\item $\boldsymbol{D^{\partial}_{\!I}}$:  the values of $v$ at the vertexes $V_i$  for $i=1, \dots, N_E$ of the element $E$,
\item $\boldsymbol{D^{\partial}_{\!I\!I}}$: the values of $v$ at $k-1$ internal points of the $(k+1)$-point  Gauss-Lobatto quadrature rule of every straight edge $e_2, \dots, e_{N_E} \in \partial E$,
\item $\boldsymbol{D^{\partial}_{\!I\!I\!I}}$: the values of $v$ at $k-1$ internal points of $e_1$ that are the images through $\gamma$ of the $k-1$ internal points of the  $(k+1)$-point  Gauss-Lobatto quadrature on $I_{e_1}$,
\item $\boldsymbol{D^o}$: the internal moments of $v$ against a  polynomial basis $\{m_i\}_{i=1}^{k(k-1)/2}$ of $\Pk_{k-2}(E)$
\[
\boldsymbol{D}^{\boldsymbol{o}}_i := \frac{1}{|E|} \int_E v \, m_i \, {\rm d}E, \qquad \text{with $\|m_i\|_{L^{\infty}(E)} = 1$.}
\]
\end{itemize}
\end{proposition}
\begin{proof}
We only sketch the very simple proof. 
The number of DoFs $\boldsymbol{D}$ is obviously equal to the dimension of the space $V_h^E$ in \eqref{eq:virtualpace}.
We need only to prove that $\boldsymbol{D^{\partial}_{\!I}}(v) = \mathbf{0}$ and $\boldsymbol{D^{\partial}_{\!I\!I\!I}}(v) = \mathbf{0}$ imply that $v_{|e_1} = 0$ since the rest of the proof follows standard VEM arguments \cite{serendipity}. 
Let $\{x_j\}_{j=1}^{k+1}$ denote the  points of the  $(k+1)$-point  Gauss-Lobatto quadrature on $I_{e_1}$,
then from definitions \eqref{eq:ptilde} and \eqref{eq:virtualpace} we have
$$
v_{|e_1}(\gamma({x_i})) =  \widetilde{q_k}(\gamma({x_i})) = \widehat{q_k}(x_i) = 0 \qquad \text{for $i=1, \dots, k+1$,} 
$$
that implies $\widehat{q_k} = 0$, and thus $v_{|e_1} = 0$. 
\end{proof}
\begin{figure}[!h]
\center{
\hspace*{-1.5cm}%
\includegraphics[scale=0.33]{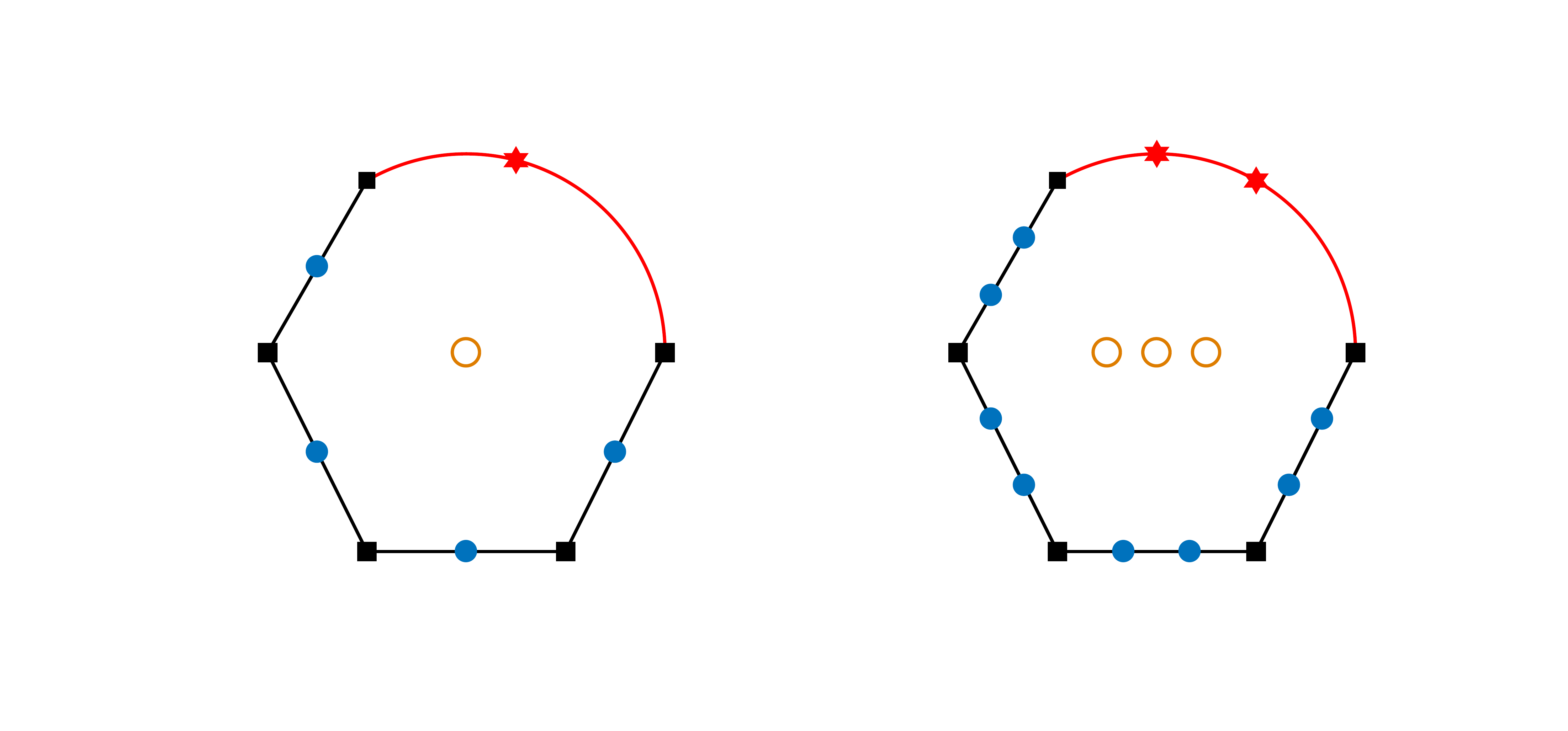} 
\vspace*{-1cm}
\caption{DoFs for $k=2$ (left) and $k=3$ (right). We denote $\boldsymbol{D^{\partial}_{\!I}}$ with the squares, $\boldsymbol{D^{\partial}_{\!I\!I}}$ with the dots,   $\boldsymbol{D^{\partial}_{\!I\!I\!I}}$ with  the hexagrams, $\boldsymbol{D^o}$ with the circles.}
\label{fig:dofsloc}
}
\end{figure}

\noindent
For future reference, 
we collect all the $k\, N_E$ boundary DoFs (the first three items above) and denote them with $\boldsymbol{D^{\partial}} := \{\boldsymbol{D}^{\boldsymbol{\partial}}_i\}_{i=1}^{k N_E}$.
Moreover 
we denote with $\boldsymbol{D^{e}} := \{\boldsymbol{D}^{\boldsymbol{e}}_i\}_{i=1}^{k+1}$ the DoFs $\boldsymbol{D^{\partial}}$ relative to the (closed) edge $e$.
For any $n \in \N$ and $E \in \Omega_h$ we introduce the following useful polynomial projections:
\begin{itemize}

\item the $H^1$ \textbf{semi-norm projection} ${\Pi}_{n}^{\nabla,E} \colon H^1(E) \to \Pk_n(E)$, defined by 
\begin{equation}
\label{eq:Pn_k^E}
\left\{
\begin{aligned}
& \int_E \nabla \,q_n \cdot \nabla (v - \, {\Pi}_{n}^{\nabla,E}   v) \, {\rm d} E = 0 \qquad  \text{for all $v \in H^1(E)$ and for all $q_n \in \Pk_n(E)$,} \\
& \int_{\partial E}  (v - \, {\Pi}_{n}^{\nabla,E}   v) \, {\rm d}S = 0 \,.
\end{aligned}
\right.
\end{equation} 

\item the $\boldsymbol{L^2}$\textbf{-projection} $\Pi_n^{0, E} \colon L^2(E) \to \Pk_n(E)$, given by
\begin{equation}
\label{eq:P0_k^E}
\int_E q_n (v - \, {\Pi}_{n}^{0, E}  v) \, {\rm d} E = 0 \qquad  \text{for all $v \in L^2(E)$  and for all $q_n \in \Pk_n(E)$.} 
\end{equation}
\end{itemize}
The following result extends to curved elements the analogous result on straight elements (see \cite{volley}). 
\begin{proposition}[Projections and Computability]
\label{prp:projections}
The DoFs $\boldsymbol{D}$ allow to compute exactly 
\[
\PN \colon V_h^E \to \Pk_k(E), \qquad
\P0 \colon V_h^E \to \Pk_{k-2}(E),
\]
in the sense that, given any $v_h \in V_h^E$, we are able to compute the polynomials
$\PN v_h$ and $\P0 v_h$ only using, as unique information, the degree of freedom values $\boldsymbol{D}$ of $v_h$.
\end{proposition}

\begin{remark}
\label{rm:integration}
In Proposition \ref{prp:projections} we are tacitly assuming that we can compute the integrals of polynomials on curved elements   and on curved edges with a given parametrization $\gamma$. 
More information on the adopted integration rules can be found in Section \ref{sec:4}.
\end{remark}
\begin{remark}
\label{rm:outside}
We observe that contrary to the straight case, the projection falls outside the virtual space $V_h^E$ (see Remark \ref{rm:non inclusion}).
\end{remark}
Finally we define the global virtual element space as
\begin{equation}
\label{eq:V_h}
V_h := \{ v \in H^1_0(\Omega)  \quad \text{s.t} \quad v_{|E} \in V_h^E  \quad \text{for all $E \in \Omega_h$} \}
\end{equation} 
with the obvious associated sets of global degrees of freedom. A simple computation shows that:
\begin{equation*}
\dim(V_h) = n_P \, \frac{k(k-1)}{2} 
+ n_V + (k-1) n_e
\end{equation*}
where $n_P$, $n_e$, $n_V$ are respectively the number of elements, internal edges and vertexes in $\Omega_h$.

\begin{remark}
\label{rm:enhanced}
Using the same ideas of \cite{Ahmed-et-al:2013}, it would also be possible to introduce a slightly different virtual space $V_h^E$ (that will not be used in the present work) to be used in place of the original one defined in \eqref{eq:virtualpace} in order to
 improve the results of Proposition \ref{prp:projections} and compute exactly also the following higher order projection 
\[
\Pi_k^{0, E} \colon V_h^E \to \Pk_{k}(E).
\] 
\end{remark}

\subsection{Discrete bilinear forms and load term approximation}
\label{sub:2.3}

The next step in the construction of our method is to define a discrete version of the gradient-gradient form $a(\cdot, \cdot)$.
First of all we decompose into local contributions the bilinear form $a(\cdot, \, \cdot)$ by defining 
\begin{equation*}
a (u, \, v) =: \sum_{E \in \Omega_h} a^E (u, \, v) \qquad \text{for all $u, v \in {V}$.}
\end{equation*}
We note that for an arbitrary pair $(u, \, v) \in V_h^E \times V_h^E $, the quantity $a^E(u, \, v)$ is clearly not computable.  Therefore, following the usual procedure in the VEM setting, we define a computable discrete local bilinear form (being careful that in general $\Pk_k(E) \not \subset V_h^E$)
\begin{equation}
\label{eq:a_h^E} 
a_h^E(\cdot, \cdot) \colon [V_h^E + \Pk_k(E)] \times  [V_h^E + \Pk_k(E)] \to \R
\end{equation}
approximating the continuous form $a^E(\cdot, \, \cdot)$, and defined by
\begin{equation}
\label{eq:a_h^E def}
a_h^E(u, \, v) := a^E \left(\PN u, \, \PN v \right) + \mathcal{S}^E \left((I - \PN) u, \, (I -\PN) v \right)
\end{equation}
for all $u, v \in V_h^E$, where the (symmetric) stabilizing bilinear form 
$$\mathcal{S}^E \colon [V_h^E + \Pk_k(E)] \times  [V_h^E + \Pk_k(E)] \to \R$$
is defined by \cite{volley,2016stability}
\begin{equation}
\label{eq:s^e}
\mathcal{S}^E (\eta, \, \sigma) := \sum_{i=1}^{k(k-1)/2} \boldsymbol{D}^{\boldsymbol{o}}_i(\eta)\boldsymbol{D}^{\boldsymbol{o}}_i(\sigma) +  \sum_{i=1}^{k N_E}  \boldsymbol{D}^{\boldsymbol{\partial}}_i(\eta)\boldsymbol{D}^{\boldsymbol{\partial}}_i(\sigma)
\end{equation} 
for all $\eta, \, \sigma \in [V_h^E + \Pk_k(E)]$.

For the stabilizing form $\mathcal{S}^E (\cdot, \, \cdot)$  it can be possible to consider also other options (see \cite{2016stability}).
%
%
We define the global approximated bilinear form 
$$a_h(\cdot, \cdot) \colon \left[V_h + \Pi_{E \in \Omega_h} \Pk_k(E) \right] \times \left[V_h + \Pi_{E \in \Omega_h} \Pk_k(E) \right] \to \R$$
by simply summing the local contributions:
\begin{equation}
\label{eq:a_h}
a_h(u,\,  v) := \sum_{E \in \Omega_h}  a_h^E(u, \, v) \qquad \text{for all $u, \, v \in \left[V_h + \Pi_{E \in \Omega_h} \Pk_k(E) \right]$.}
\end{equation}

The last step consists in constructing a computable approximation of the right-hand side $(f, \, v)$ in \eqref{eq:continua}.  We define the approximated load term $f_h$ (see \cite{volley,Ahmed-et-al:2013}) for $k \geq 2$ as 
\begin{equation} 
\label{eq:f_h}
f_h := \Pi_{k-2}^{0,E} f \qquad \text{for all $E \in \Omega_h$,}
\end{equation}
and consider:
\begin{equation}
\label{eq:right}
(f_h, v)  = \sum_{E \in \Omega_h} \int_E f_h \, v \, {\rm d}E = \sum_{E \in \Omega_h} \int_E \Pi_{k-2}^{0,E}f \, v \, {\rm d}E = \sum_{E \in \Omega_h} \int_E f \, \Pi_{k-2}^{0,E}  v \, {\rm d}E,
\end{equation}
whereas for $k=1$ we approximate $f$ by a piecewise constant, and consider:
\begin{equation}
\label{eq:right1}
(f_h, v)  = \sum_{E \in \Omega_h} \left( |E| \, \left(\Pi_0^{0, E} f \right) \, \frac{1}{N_E} \sum_{i=1}^{N_E}  \boldsymbol{D}^{\boldsymbol{\partial}}_i(v) \right)\,.
\end{equation}
We observe that \eqref{eq:right} and \eqref{eq:right1} can be exactly computed from $\boldsymbol{D}$ for all $v \in V_h$ (see Proposition \ref{prp:projections}). 


%

%

%


\subsection{Virtual Element Problem}
\label{sub:2.4}

We are now ready to state the proposed discrete problem. Referring to~\eqref{eq:V_h}, ~\eqref{eq:a_h} and either ~\eqref{eq:right} or \eqref{eq:right1}, we consider the \textbf{virtual element problem}:
\begin{equation}
\label{eq:virtual}
\left\{
\begin{aligned}
& \text{find $u_h \in V_h$, such that} \\
& a_h(u_h, \, v_h)  = (f_h, \, v_h) \qquad & \text{for all $v_h \in V_h$.} \\
\end{aligned}
\right.
\end{equation}
By construction the bilinear $a_h(\cdot,\,  \cdot)$ is symmetric. Therefore, the existence and the uniqueness of the solution to Problem \eqref{eq:virtual} will follow if $a_h(\cdot,\,  \cdot)$  is also  coercive on $V_h$, which is investigated in Section \ref{sub:stability}.

\begin{remark}
It is worth to point out that
if $\Omega$ is a straight polygon (i.e. if $\Gamma$ is made up of a finite number of straight sides), we recover the standard VEM  \cite{volley}.
\end{remark}

\section{Theoretical analysis}
\label{sec:3}

\subsection{Interpolation estimates}
\label{sub:3.1}
In this section we prove that the virtual space on the curved element $V_h$ (cf. \eqref{eq:V_h}) 
has the optimal approximation order (also in higher order norms).
We start our analysis with the following technical lemma (simple consequence of \cite{Ciarlet:1972},  Lemma 3).
\begin{lemma}
\label{lm:composition}
Let $U$, $V$, $Z \subset \R$ be three open intervals and $f \colon U \to V$, $g \colon V \to Z$ two mappings with $f \in W^{m, \infty}(U)$ and $g \in H^m(V)$. Then the mapping $h = g \circ f  \colon U \to Z$ is  in $H^m(U)$   and
\[
|h|_{H^m(U)} \leq \alpha \, \sum_{l=1}^m  \left( |g|_{H^l(V)}  \, \sum_{\boldsymbol{i} \in I(l, m)} \|D f\|_{L^{\infty}(U)}^{i_1} \|D^2 f\|_{L^{\infty}(U)}^{i_2} \dots \|D^m f\|_{L^{\infty}(U)}^{i_m} \right)
\]
where $\alpha$ is a constant depending only upon $m$, and
\[
I(l, m) = \{\boldsymbol{i} = (i_1, i_2, \dots, i_m) \in \N^m \, | \, i_1 + i_2 + \dots + i_m = l, \, i_1 + 2 i_2 + \dots + m i_m = m\}.
\]
\end{lemma}

\begin{lemma}
\label{lm:edgeinterpolation}
Let $e \subset \partial E$ be a (possibly curved) edge of an element $E \in \Omega_h$ and let $v \in H^s(e)$ for $1 < s \leq k+1$. Let $v_{\rm I} \in \widetilde{\Pk_k}(e)$ be the function determined by
\begin{equation}
\label{eq:boundaryinterpolant}
\boldsymbol{D^{e}} (v  - v_{\rm I}) = 0.
\end{equation}
 Then for all $0 \leq m \leq s$
\begin{equation}
\label{eq:interpolation}
|v - v_{\rm I}|_{m, e} \leq C \, h_E^{s-m} \, \|v\|_{s, e}
\end{equation}
where the constant $C$ depends on  $k$ and, if the edge is curved, on  $\|\zeta^{-1}\|_{W^{s, \infty}}$ and $\|\zeta\|_{W^{s, \infty}}$.
\end{lemma}

\begin{proof}
We focus on the case of $e$ curved edge, since the straight case is trivial.
We first suppose that $s$ is an integer.
Note that, since  $v_{\rm I} \in \widetilde{\Pk_k}(e)$, we can write $v_{\rm I} = \widehat{q_k} \circ \gamma^{-1}$, with $\widehat{q_k} \in \Pk_k(I_e)$. 
Then using the notation in \eqref{eq:correspondences} and by definition of Sobolev norm \eqref{eq:normcurves}, we get
\begin{equation}
\label{eq:zero}
|v - v_{\rm I} |_{m, e} = \left|\check{v} - \check{v_{\rm I}}\right|_{H^m((0, \, \ell_e))}  
=\left| \left( \hat{v} - \hat{v_{\rm I}} \right) \circ \zeta^{-1}  \right|_{H^m((0, \, \ell_e))} = \left| \left( \hat{v} - \widehat{q_k} \right) \circ \zeta^{-1}  \right|_{H^m((0, \, \ell_e))}.
\end{equation}
From \eqref{eq:zero} applying Lemma \ref{lm:composition}, we get
\begin{equation}
\label{eq:first}
|v - v_{\rm I} |_{m, e} 
\leq C
\sum_{l=1}^m  \left( |\hat{v} - \widehat{q_k}|_{H^l(I_e)}  \, \sum_{i \in I(l, m)} \|D \zeta^{-1}\|^{ i_1}_{L^{\infty}} \|D^2 \zeta^{-1}\|^{ i_2}_{L^{\infty}} \dots \|D^m \zeta^{-1} \|^{ i_m}_{L^{\infty}}
\right) \,.
\end{equation}
We notice now that, from \eqref{eq:boundaryinterpolant}, $\widehat{q_k}$ is the $k$-degree polynomial interpolant of $\hat{v}$ in the Gauss-Lobatto nodes. Furthermore, as we have already observed in Section \ref{sub:2.2}, it easily follows that the length of the interval $I_e$ is comparable with the diameter $h_E$ of the element $E$. 
Therefore, from \eqref{eq:first} and standard polynomial approximation results \cite{brennerscott}, we can conclude that
\begin{equation}
\label{eq:first1}
|v - v_{\rm I} |_{m, e} 
\leq
C \, h_E^{s-m} \,  |\hat{v}|_{H^{s}(I_e)}
\end{equation}
where the constant $C$   depends on $\|D \zeta^{-1}\|_{W^{m, \infty}((0, \, \ell_e))}$, that is uniformly bounded since it does not depend on the mesh (both $\gamma$ and $\zeta$ are fixed once and for all).
Using again Lemma \ref{lm:composition}  we obtain
\begin{equation}
\label{eq:third}
|\hat{v}|_{H^{s}(I_e)} \leq C 
\sum_{l=1}^s  \left( |\check{v}|_{H^l((0, \ell_e))}  \, \sum_{i \in I(l, m)} \|D \zeta\|^{ i_1}_{L^{\infty}} \|D^2 \zeta\|^{ i_2}_{L^{\infty}} \dots \|D^s \zeta \|^{ i_s}_{L^{\infty}} \right)
\end{equation}
Collecting  \eqref{eq:first1} and \eqref{eq:third} in \eqref{eq:zero} and recalling 
\eqref{eq:normcurves}  we obtain the thesis for integer $s$.\\
For non integer $s = n + \sigma$ with $n \in \N$ and $ 0< \sigma < 1$, using a standard result of interpolation theory concerning operators on Banach spaces (see Theorem 14.2.3 and Proposition 14.1.5 in \cite{brennerscott}) it follows that
\[
\|v - v_{\rm I}\|_{m, e} \leq C \, h_E^{(n - m) (1 - \sigma)} \, h_E^{ (n +1 - m) \, \sigma}\| v\|_{s, e} =  C \, h_E^{s-m}\| v\|_{s, e}.
\]
\end{proof}

\begin{remark}
We notice that if the given parametrization $\gamma$ is the arc-length parametrization of   the edge, the constant $C$ in \eqref{eq:interpolation} depends only on $k$.
\end{remark}
The following technical lemma extends to curved elements the result of Lemma 6.1 in \cite{2016stability}.
\begin{lemma}[Trace Theorem]
\label{lm:tracethm}
Let $E \in \Omega_h$. Under the assumptions $\mathbf{(A0)}$, $\mathbf{(A1)}$ and $\mathbf{(A2)}$,  
for all $v \in H^{1+\epsilon}(E)$ with $0 \leq \epsilon < \frac{1}{2}$, it holds
\begin{equation}
\label{eq:trace1}
| v|_{1/2 + \epsilon, \partial E} \leq C \, |v|_{1+ \epsilon, E} \,.
\end{equation}
The constant $C$ in the previous estimate depends only on the shape regularity constant $\rho$ and $\|\gamma'\|_{L^{\infty}}$, 
in particular $C$ does not depend on $h_E$.
\end{lemma}

\begin{proof}
By assumption $\mathbf{(A0)}$ the element $E$ is a Lipschitz domain. For the trace theorem of Sobolev spaces on Lipschitz domains \cite{costabel}, the  trace operator 
is a bounded linear operator from $H^{1+\epsilon}(E)$ to $H^{1/2 + \epsilon}(\partial E)$, $0\leq \epsilon < \frac{1}{2}$. We need only to prove the uniformity (with respect to the mesh) of the  above continuity constant. 

We only sketch the proof, because it essentially follows the guidelines of Lemma 6.1 in \cite{2016stability}.
Up to a translation of the element $E$, we may assume that the ball $B_E$  is centered in the origin of
the coordinate axes. Let $\Psi \colon [0, 2 \pi) \to [\rho h_E, h_E]$ the map describing the boundary $\partial E$ of $E$ with respect to  the angle in radial coordinates, i.e.
\[
(x, \, y) =  (\Psi(\theta) \cos \theta, \, \Psi(\theta) \sin \theta)  \in \partial E  \qquad \text{for all $\theta \in [0, 2 \pi)$.}
\]
As for the straight case,  the regularity of the curve $\mathbf{(A0)}$ and the star shaped assumption $\mathbf{(A1)}$ implies that $\Psi \in W^{1, \infty}([0, 2 \pi))$ uniformly with respect to $E \in \Omega_h$. As a matter of fact, $\Psi$ is clearly continuous piecewise differentiable, and moreover its derivative is uniformly bounded. Indeed if one assumes that the derivative of $\Psi$ has a blow up, then the element $E$ cannot be star shaped with respect to a ball of radius comparable with the diameter of the element in contrast with the assumption $\mathbf{(A1)}$  (see Figure \ref{fig:traccia}).
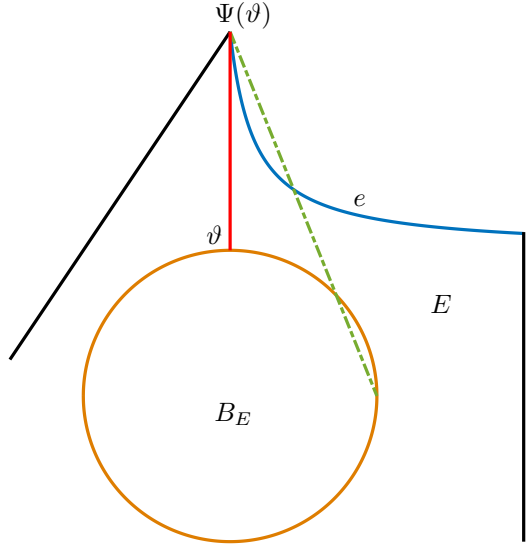
\begin{figure}[!h]
\center{
\hspace*{-2.5cm}%
\begin{overpic}[scale= 0.40]{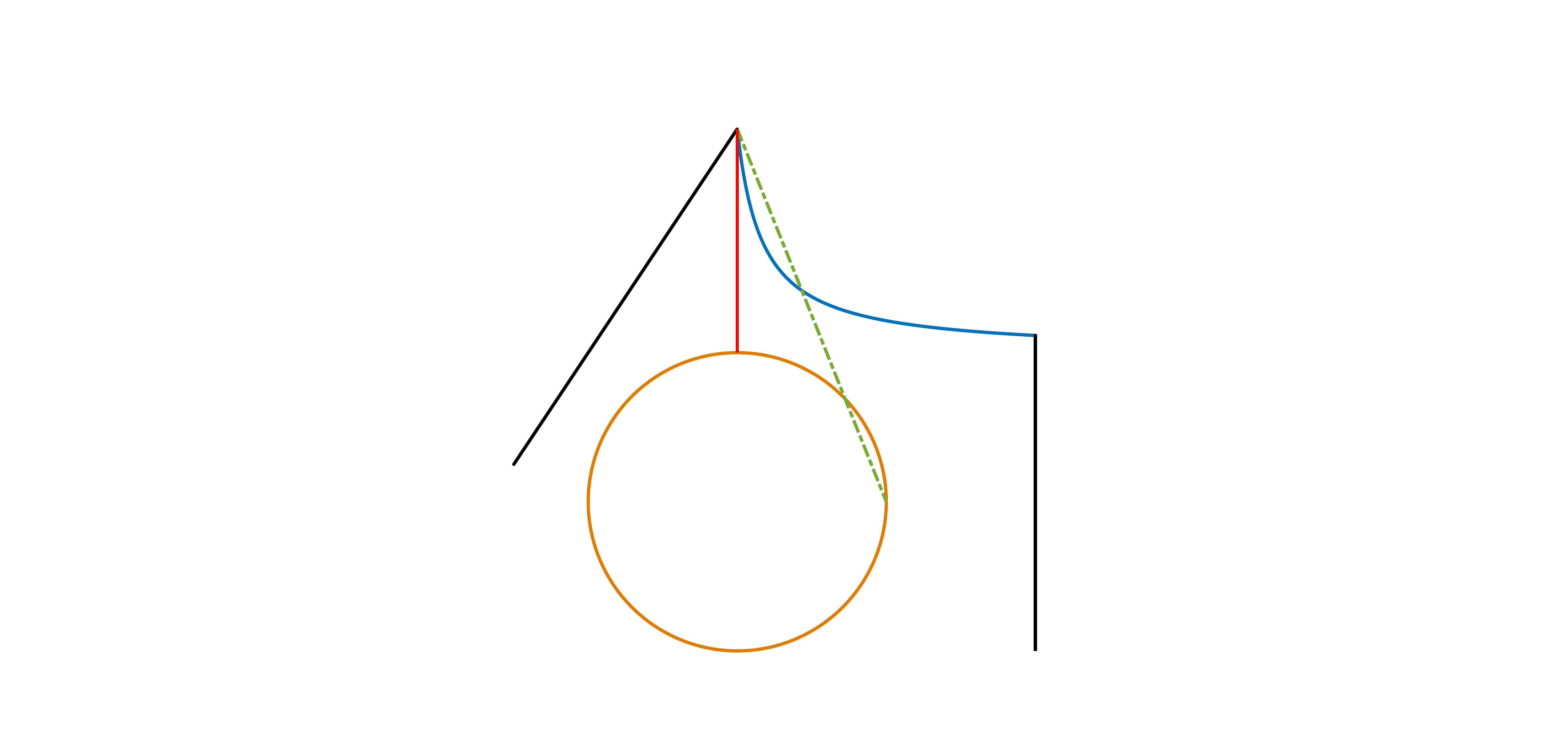} 
\put (46,39) {{$\Psi(\theta)$}}
\put (60,20) {$E$}
\put (55,27) {$e$}
\put (45.5,24.5) {$\theta$}
\put (46,13) {$B_E$}
\end{overpic}
\vspace*{-1cm}
\caption{Example of element not satisfying the assumption  $\mathbf{(A1)}$. The derivative of the function $\Psi$ being not bounded is in contrast with the element $E$ being star shaped with respect to a ball.}
\label{fig:traccia}
}
\end{figure}
Let now $F \colon \bar{B}_E \to \bar{E}$ be the radial mapping 
\[
F(\bar{x}, \, \bar{y}) = \frac{r}{\rho h_E} ( \Psi(\theta)\cos \theta, \, \Psi(\theta) \sin \theta ) \in \bar{E} \qquad \text{for all $ (\bar{x}, \, \bar{y}) = (r \cos \theta, \, r \sin \theta) \in  \bar{B}_E$.}
\]
As a consequence of the above observation, assumptions $\mathbf{(A0)}$ and $\mathbf{(A1)}$ imply that
the piecewise regular map $F \in W^{1, \infty}(\bar{B}_E)$ and $F^{-1} \in W^{1, \infty}(\bar{E})$, uniformly in the element $E$.
The thesis now follows from standard ``pull-back'' and ``push-forward'' arguments combined
with the analogous result on the disk $B_E$ (see for instance \cite{brennerscott, ding1996proof}).%
\end{proof}

\begin{remark}
\label{rm:pluto}
A simple consequence of the proof of Lemma \ref{lm:tracethm} is also that all (internal) angles of any element $E \in \Omega_h$ are uniformly separated from $0$ and $2  \pi$.
\end{remark}

Let us recall a classical approximation result for  polynomials on star-shaped domains, see for instance \cite{brennerscott}.

\begin{lemma}
\label{lm:scott}
Let $E \in \Omega_h$, and let two real non-negative numbers $r$, $s$ with $r \leq s \leq k+1$.
Then for all $v \in H^s(E)$, there exists a  polynomial function $v_{\pi} \in \Pk_k(E)$,  such that
\begin{equation}
\label{eq:scott}
| v - v_{\pi} |_{r, E} \leq C \, h_E^{s-r}| v|_{s, E},
\end{equation}
with $C$ depending only on $k$ and the shape regularity constant $\rho$ (cf. assumption $\mathbf{(A1)}$).
\end{lemma}

\begin{theorem}
\label{thm:interpolante}
Let any real number $\epsilon \in [0, 1/2)$ and
$v \in H^{s}(\Omega) \cap V$, with $1+\epsilon  < \frac{3}{2} \leq s \leq k+1$. Then there exists $v_{\rm I} \in V_h$ such that
\[
\sum_{E \in \Omega_h} | v - v_{\rm I} |_{1 + \epsilon, E}  \leq C \, h^{s-1- \epsilon} \, \|v\|_{s},
\]
where the constant $C$ depends on the degree $k$, the shape regularity constant $\rho$ and the 
parametrization $\gamma$.
\end{theorem}
\begin{proof} 
For any element $E \in \Omega_h$ 
Lemma \ref{lm:scott} yields a polynomial $v_{\pi} \in \Pk_k(E)$ such that
\begin{equation}
\label{eq:inter_new}
|v - v_{\pi}|_{1+ \epsilon, E} \leq  C \, h_E^{s-1- \epsilon}|v|_{s, E} \,.
\end{equation}
Let us define the following elliptic problem 
\begin{equation}
\label{eq:pbinter}
\left\{
\begin{aligned}
& \text{Find $v_{\rm I} \in H^1(E)$, such that} \\
& \Delta v_{\rm I}  =  \, \Delta v_{\pi} \qquad & \text{in $E$,} \\
&  v_{\rm I} = \text{interpolant of $v$ in the sense of \eqref{eq:boundaryinterpolant}} \qquad & \text{ on $\partial E$.}
\end{aligned}
\right.
\end{equation}
Note that $v_{\rm I} \in V_h^E$ and the difference $(v_{\rm I} - v_{\pi})$ satisfies
\begin{equation}
\label{eq:pbinterdiff}
\left\{
\begin{aligned}
& \Delta (v_{\rm I}  - v_{\pi})= 0 \qquad & \text{in $E$,} \\
&  (v_{\rm I} - v_{\pi})  \qquad & \text{prescribed on $\partial E$.}
\end{aligned}
\right.
\end{equation}
Due to assumptions $\boldsymbol{(A0)}$-$\boldsymbol{(A1)}$ and regularity results 
on star-shaped Lipschitz domains \cite{Agmon:1965,Grisvard:1992},
the solution of the elliptic problem \eqref{eq:pbinterdiff} belongs to the space $H^{1+\epsilon}(E)$ with continuous dependence from the boundary data.  
Therefore, also using inequalities  \eqref{eq:trace1} and \eqref{eq:inter_new}, we obtain
\begin{equation}
\label{eq:interp1}
\begin{split}
|v_{\pi} - v_{\rm I} |_{1+ \epsilon, E}
& \lesssim |v_{\pi} - v_{\rm I} |_{1/2+ \epsilon, \partial E}
  \lesssim |v_{\pi} - v |_{1/2+ \epsilon, \partial E} + |v - v_{\rm I} |_{1/2+ \epsilon, \partial E} \\  
& \lesssim |v_{\pi} - v |_{1+ \epsilon, E} + |v - v_{\rm I} |_{1/2+ \epsilon, \partial E} \\
& \lesssim h^{s-1-\epsilon} \, |v|_{s, E} + |v - v_{\rm I} |_{1/2+ \epsilon, \partial E} \,.
\end{split}
\end{equation}
The uniformity (with respect to the element $E$) of the fist bound in \eqref{eq:interp1} follows from the fact that both the star-shapedness constant and the Lipschitz constant of all elements $E \in \Omega_h$ are uniformly bounded due to  $\boldsymbol{(A0)}$ and $\boldsymbol{(A1)}$, see also Remark \ref{rm:pluto}.
The second factor in the right hand side of \eqref{eq:interp1} is estimated as follows.  
For all edges $e \subset \partial E$, let $\zeta_e$ be the function defined on $\partial E$ by
\begin{equation}
\label{eq:interp6}
\zeta_e :=  
\left\{
\begin{aligned}
& v - v_{\rm I} \qquad & \text{in $e$,} \\
& 0  \qquad & \text{in $\partial E \setminus e$.}
\end{aligned}
\right.
\end{equation}
It is straightforward to see that 
\begin{equation}
\label{eq:interp5}
(v - v_{\rm I} )_{|_{\partial E}} = \sum_{e \in \partial E} \zeta_e
\end{equation}
and that, being $(v - v_{\rm I})$ zero at all vertexes of $E$,  for all $e \in \partial E$ it holds $\zeta_e \in H^1(\partial E)$. 

Now, using the characterization of fractional Sobolev space $H^{1/2+ \epsilon}( \partial E)$ as the real interpolation between $L^2(\partial E)$ and $H^1(\partial E)$, by a standard result concerning the norm of  real interpolation  spaces (Proposition 2.3 of \cite{lions-magenes}) it holds that
\begin{equation}
\label{eq:pluto}
|\zeta_e|_{1/2+\epsilon, \partial E} \leq  \, \|\zeta_e\|_{0, \partial E}^{1/2- \epsilon} \, \|\zeta_e\|_{1, \partial E}^{1/2+\epsilon}.
\end{equation}
Now, applying the 1D Poincar{\'e} inequality $\|\zeta_e\|_{0, \partial E} \lesssim \, h_E \, |\zeta_e |_{1, \partial E}$, we get
\begin{equation}\label{eq:L:X}
|\zeta_e|_{1/2+\epsilon, \partial E} \lesssim \, h_E^{1/2-\epsilon} \,  \|\zeta_e\|_{1, \partial E}^{1/2-\epsilon} \, \|\zeta_e\|_{1, \partial E}^{1/2+ \epsilon} 
\\
\lesssim \, h_E^{1/2- \epsilon} \,  \|\zeta_e\|_{1, e} 
\lesssim \, h_E^{1/2 - \epsilon} \,  \|v - v_{\rm I}\|_{1, e} . \\
\end{equation}
From the above bound \eqref{eq:L:X} we now distinguish two cases. If the edge $e$ is straight then standard polynomial approximation estimates in 1D yield 
$$
|\zeta_e|_{1/2 + \epsilon, \partial E} \lesssim \, h_E^{s-1 -\epsilon} \, |v |_{s - 1/2, e}. 
$$
By assumption $\boldsymbol{(A0)}$ the boundary $\partial E$ is Lipschitz. Therefore the above bound and a standard trace inequality (which we note is here applied to a single straight edge) yields
\begin{equation}\label{S-case}
|\zeta_e|_{1/2 + \epsilon, \partial E} \lesssim \, h_E^{s-1 - \epsilon} \, |v|_{s, E} .
\end{equation}
Referring back to bound \eqref{eq:L:X}, if the edge $e$ is curved (in other words $e \subseteq \Gamma$), using Lemma \ref{lm:edgeinterpolation} we obtain
\[
\begin{split}
|\zeta_e|_{1/2+ \epsilon, \partial E} & \lesssim \, h_E^{1/2 - \epsilon} \,  h_E^{s-3/2} \,\|v\|_{s - 1/2, e} \lesssim \, h_E^{s-1 - \epsilon} \,\|v \|_{s - 1/2, e} .
\end{split}
\]
Therefore \eqref{eq:interp5} and the bounds above give
\begin{equation}
\label{eq:interp8}
|v - v_{\rm I}|_{1/2 + \epsilon, \partial E} \leq 
\sum_{e \in \subset \partial E}  \,|\zeta_e |_{1/2+\epsilon, e} \lesssim
 \, h_E^{s-1- \epsilon} \big( |v|_{s, E} + \|v \|_{s - 1/2, \partial E \cap \Gamma} \big).
\end{equation}
First combining the previous inequality with \eqref{eq:interp1}, then summing over all the elements $E \in \Omega_h$, yields
$$
\sum_{E \in \Omega_h}| v - v_{\rm I} |_{1+\epsilon, E} \lesssim \, h^{s-1-\epsilon} \, \big( |v|_{s} + \sum_{i=1}^ N \|v \|_{s - 1/2, \Gamma_i}  \big) .
$$
Now the thesis follows from standard (piecewise) trace inequalities on Lipschitz domains with piecewise regular boundary.
More in details, by assumption $\boldsymbol{(A0)}$, the boundary $\partial \Omega$ is Lipschitz.
Therefore applying the Extension Theorem for a domain with Lipschitz boundary, (see for instance Theorem $5'$ in Chapter VI of \cite{stein}) there exists an extension operator
\begin{equation}
\begin{gathered}
\label{eq:extension}
\mathcal{E} \colon H^{s}(\Omega) \to H^s(\R^2) \qquad \text{s.t} \\
\mathcal{E}(u)_{|\Omega} = u 
\qquad \text{and} \qquad
\|\mathcal{E}(u)\|_{s, \R^n} \leq \kappa \, \|u\|_{s} 
\quad \text{for all $u \in H^s(\Omega)$}
\end{gathered}
\end{equation}
where the constant $\kappa$ depends on $s$.
For any curve $\Gamma_i \subset \partial \Omega$, let $\mathcal{C}_i$ be a domain in $\R^2$, with boundary $\partial \mathcal{C}_i \in C^{s-1, 1}$, such that $\Gamma_i \subset \partial \mathcal{C}_i$. Therefore applying the trace theorem for smooth domains \cite{lions-magenes}, and by \eqref{eq:extension},  we get
\begin{equation}
\label{eq:interp9}
\begin{split}
\|v\|_{s-1/2, \Gamma_i} &= \|\mathcal{E}(u)\|_{s-1/2, \Gamma_i}  \leq
\|\mathcal{E}(u)\|_{s-1/2, \partial \mathcal{C}_i} \leq 
\|\mathcal{E}(u)\|_{s, \mathcal{C}_i} \leq
\|\mathcal{E}(u)\|_{s, \R^n} \leq
\kappa \, \|u\|_{s}.
\end{split}
\end{equation}
\end{proof} 

\begin{remark}
\label{rm:piddu}
The condition $s \geq \frac{3}{2}$ could be reduced to $s > 1 + \epsilon$ by a finer (but slightly more complicated) choice of the real interpolation space in \eqref{eq:pluto}; we here prefer to be slightly less general in favour of readability.
\end{remark}

\subsection{Stability analysis}
\label{sub:stability}

In the present section we address the stability properties of the bilinear form introduced in \eqref{eq:a_h^E def}. 
Although we analyse the standard choice of the stabilizing bilinear form defined in \eqref{eq:s^e}, the same results hold for different choices of $\mathcal{S}^E(\cdot, \, \cdot)$ in \eqref{eq:a_h^E def} (see for instance \cite{2016stability}).
%
%
%
%

The next lemma extends to curved elements the inverse inequalities of Lemma 7.1 in \cite{vacca2017} and Lemma 6.3 in \cite{2016stability}, see also
 \cite{brenner2017, chen2017}. The proof is identical to the straight case and is omitted.
\begin{lemma}
\label{lm:inverse estimate}
Let $E \in \Omega_h$. Under the assumptions $\mathbf{(A0)}$, $\mathbf{(A1)}$, $\mathbf{(A2)}$, for all $v_h \in V_h^E$  the following inverse inequalities hold
\begin{align}
\label{eq:inv}
|v_h|_{1,E}  \leq c_{inv} \, h_E^{-1} \, \|v_h\|_{0,E} 
\\
\label{eq:inv2}
\|\Delta v_h \|_{0,E}  \leq c_{inv} \, h_E^{-1} \, |v_h|_{1,E} 
\end{align}
where the constant $c_{inv}$ depends only on the shape regularity constant $\rho$.
\end{lemma}
%
%

%
%
The following lemma states, in some sense, the continuity of the stabilizing form  $\mathcal{S}^E(\cdot, \, \cdot)$.
%
%

\begin{lemma}
\label{lm:continuity}
Let $E \in \Omega_h$. Under the assumptions $\mathbf{(A0)}$, $\mathbf{(A1)}$ and $\mathbf{(A2)}$ 
for any $\epsilon \in (0, 1/2)$
the following holds
\begin{gather}
\label{eq:continuityvp}
\mathcal{S}^E( \eta, \, \eta) \leq C\,  h_E^{-2} \, \|\eta \|_{0,E}^2 +
C \, h_E^{2 \epsilon} \, |\eta|^2_{1 + \epsilon,  E}   \qquad 
\text{for all $\eta \in V_h^E + \Pk_k(E)$,}
\end{gather}
where the constant $C$ depends on $k$, $\epsilon$, the shape regularity constant $\rho$ and $D\zeta$.
\end{lemma}

\begin{proof}
By definition \eqref{eq:s^e} we get
\begin{equation}
\label{eq:3.1}
\begin{split}
\mathcal{S}^E (\eta, \, \eta) &= \sum_{i=1}^{k(k-1)/2} \boldsymbol{D}^{\boldsymbol{o}}_i(\eta)^2 +  \sum_{i=1}^{k N_E}  \boldsymbol{D}^{\boldsymbol{\partial}}_i(\eta)^2 
\qquad \text{for all $\eta \in V_h^E + \Pk_k(E)$.}
\end{split}
\end{equation}
For what concerns the first term in \eqref{eq:3.1}, recalling that the polynomial basis  $\{m_i\}_{i=1}^{k(k-1)/2}$ in $\boldsymbol{D^o}$ is such that $\|m_i\|_{L^{\infty}(E)} \leq 1$,  
Cauchy-Schwarz inequality yields 
\begin{equation}
\label{eq:3.2}
\begin{split}
\sum_{i=1}^{k(k-1)/2} \boldsymbol{D}^{\boldsymbol{o}}_i(\eta)^2 
&=\sum_{i=1}^{k(k-1)/2} \left( \frac{1}{|E|} \int_E \eta \, m_i \, {\rm d}E  \right)^2 
  \leq \sum_{i=1}^{k(k-1)/2} \frac{1}{|E|^2} \|\eta\|_{0,E}^2 \, \|m_i\|_{0,E}^2 \\
& \leq \sum_{i=1}^{k(k-1)/2} \frac{1}{|E|} \|\eta\|_{0,E}^2  
\lesssim  h_E^{-2} \, \|\eta\|_{0,E}^2 
\,.
\end{split}
\end{equation}
The estimate of the second term in \eqref{eq:3.1} easily follows by a scaled Sobolev's inequality, since the space $L^{\infty}(E)$ is continuously embedded in $H^{1 + \epsilon}(E)$:
\begin{equation}
\label{eq:3.3}
\begin{split}
\sum_{i=1}^{k N_E}  \boldsymbol{D}^{\boldsymbol{\partial}}_i(\eta)^2
& \leq k N_E \, \| \eta\|^2_{L^{\infty}(\partial E)}
  \lesssim   \|\eta\|^2_{L^{\infty}(E)}
  \lesssim  h_E^{-2} \, \|\eta\|_{0,E}^2 + h_E^{2 \epsilon} \, |\eta|_{1+\epsilon, E}^2  \,.
\end{split}
\end{equation}
The uniformity (with respect to the element $E$) of the last bound in \eqref{eq:3.3} can be derived by the same identical argument in Lemma \ref{lm:tracethm} (map to the ball, apply the result, map back).  
Bound \eqref{eq:continuityvp} follows collecting \eqref{eq:3.2} and \eqref{eq:3.3} in \eqref{eq:3.1}.

\end{proof}

The  next lemmas state the coercivity of the bilinear form $\mathcal{S}^E(\cdot, \, \cdot)$ with respect to  the $H^1(E)$ semi-norm. 
We start by noting that any function $v_h \in V_h^E$, can be decomposed as
\begin{equation}
\label{eq:v1v2}
v_h = v_1 + v_2
\end{equation}
where $v_1$, $v_2 \in V_h^E$ are defined by
\begin{equation}
\label{eq:v1v2a}
\left \{
\begin{aligned}
&\Delta v_1 = 0 \, & \text{in $E$,}  \\
& {v_1}_{|\partial E} = {v_h}_{|\partial E} \, & \text{on $\partial E$,}
\end{aligned} 
\right.
\qquad
\text{and}
\qquad
\left \{
\begin{aligned}
&\Delta v_2 = \Delta v_h \, & \text{in $E$,} \\
& {v_2}_{|\partial E} = 0 \, & \text{on $\partial E$.}
\end{aligned}
\right.
\end{equation}
Moreover the decomposition is $H^1$-orthogonal, i.e.
\begin{equation}
\label{eq:pitagora}
|v_h|^2_{1, E} = |v_1|^2_{1, E} + |v_2|^2_{1, E} \, .
\end{equation}
%
\noindent
Given a vector $\mathbf{g} := (g_i)_{i=1}^N$, we introduce the norm
\[
\|\mathbf{g}\|^2_{l^2} := \sum_{i=1}^N g_i^2 \,.
\]
The following  lemma for polynomials is easy to check (the simple proof is omitted).
\begin{lemma}
\label{lm:l2piccolo}
Let $E \in \Omega_h$. Under the assumptions $\mathbf{(A0)}$, $\mathbf{(A1)}$ and $\mathbf{(A2)}$,
let $\mathbf{g}:= (g_i)_{i=1}^{k(k-1)/2}$ be a vector of real numbers and $g := \sum_i^{k(k-1)/2} g_i \, m_i \in \Pk_{k-2}(E)$. 
Then we have the following norm equivalence
\[
\beta_* \, h_E^2 \, \|\mathbf{g}\|^2_{l^2} \leq \|g\|^2_{0, E} \leq \beta^* \, h_E^2 \, \|\mathbf{g}\|^2_{l^2}
\]
for two positive uniform constants $\beta_*$, $\beta^*$.
\end{lemma}


\begin{lemma}
\label{lm:coercivity}
Let $E \in \Omega_h$. Under the assumptions $\mathbf{(A0)}$, $\mathbf{(A1)}$ and $\mathbf{(A2)}$ the following holds
\[
\mathcal{S}^E(v_h, \, v_h) \geq C \,  |v_h|^2_{1, E} \qquad \text{for all $v_h \in V_h^E$}
\]
where the constant $C$ depends on $k$ and the shape regularity constant $\rho$.
\end{lemma}

\begin{proof}
The proof makes use of tools from \cite{2016stability, chen2017, vacca2017}. 
%
%
Let $v_h = v_1 + v_2$ be the $H^1$-orthogonal decomposition given by  \eqref{eq:v1v2}. 
%
We start by analysing the term $v_1$. 
Being $v_1 \in V_h^E$, by Lemma \ref{lm:inverse estimate} we get
\[
|v_1|_{1, E}  \lesssim  h_E^{-1} \, \|v_1\|_{0, E} \, , 
\]
moreover by H\"{o}lder inequality it holds that
\[
\|v_1\|_{0, E}  \leq |E|^{1/2} \, \|v_1\|_{L^{\infty}(E)} \lesssim h_E \, \|v_1\|_{L^{\infty}(E)}.
\]
Now, collecting the previous inequalities and applying the maximum principle to the harmonic function $v_1$ and recalling that ${v_1}_{|\partial E} = {v_h}_{|\partial E}$, we get 
\begin{equation}
\label{eq:4.2a}
\begin{split}
|v_1|^2_{1, E} 
& \lesssim   \|v_1\|^2_{L^{\infty}(E)} 
\lesssim  \|v_1\|^2_{L^{\infty}(\partial E)}
\lesssim  \|v_h\|^2_{L^{\infty}(\partial E)} \,.
\end{split}
\end{equation}
Finally, standard results for polynomials in one dimension (notice that on the curved edge $e$ we work on the correspondent straight segment $I_e$), we get
\begin{equation}
\label{eq:4.2b}
\|v_h\|^2_{L^{\infty}(\partial E)}
\lesssim \sum_{i=1}^{k N_E}   \boldsymbol{D}^{\boldsymbol{\partial}}_i(v_h)^2
\lesssim \mathcal{S}^E(v_h, \, v_h) 
\end{equation}  
so that from \eqref{eq:4.2a} and \eqref{eq:4.2b} we obtain
\begin{equation}
\label{eq:4.2}
|v_1|^2_{1, E} 
 \lesssim \mathcal{S}^E(v_h, \, v_h) \,.
\end{equation}
Note that, as a by product of the above calculations, we also have
\begin{equation}
\label{eq:lorenzo}
h_E^{-1} \, \|v_1\|_{0, E} \lesssim \mathcal{S}^E(v_h, \, v_h) \,.
\end{equation}
For what concerns the term $v_2$, recalling \eqref{eq:v1v2a}, integration by parts yields
\begin{equation}
\label{eq:4.3}
\begin{split}
|v_2|^2_{1, E} 
& = - \int_E \Delta v_2 \, v_2 \, {\rm d}E   
  = - \int_E \Delta v_2 \, v_h \, {\rm d}E +  \int_E \Delta v_2 \, v_1 \, {\rm d}E \,. 
\end{split}
\end{equation}
Let us analyse the first integral in the right hand side of \eqref{eq:4.3}.
Being $\Delta v_2 \in \Pk_{k-2}(E)$, we can write
\[
\Delta v_2 = \sum_{i=1}^{k(k-1)/2} g_i \, m_i
\] 
therefore by definition of $\boldsymbol{D^o}$ and H\"{o}lder inequality for sequences we get
\begin{equation}
\label{eq:4.4a}
\begin{split}
\int_E \Delta v_2 \, v_h \, {\rm d}E  
&= \sum_{i=1}^{k(k-1)/2} g_i \int_E m_i \, v_h \, {\rm d}E 
 = \sum_{i=1}^{k(k-1)/2} |E| \, g_i \, \boldsymbol{D}^{\boldsymbol{o}}_i(v_h) \\
&\leq |E| \, \|\mathbf{g}\|_{l^2} \, \left( \sum_{i=1}^{k(k-1)/2} \boldsymbol{D^{o}_i}(v_h)^2 \right)^{1/2}
\lesssim h_E^2 \, \|\mathbf{g}\|_{l^2} \, \left(\mathcal{S}^E(v_h, \, v_h) \right)^{1/2} \,.
\end{split}
\end{equation}
From \eqref{eq:4.4a}, using Lemma \ref{lm:l2piccolo} and \eqref{eq:inv2} we obtain
\begin{equation}
\label{eq:4.4}
\begin{split}
\int_E \Delta v_2 \, v \, {\rm d}E  
& \lesssim h_E \, \|\Delta v_2\|_{0, E} \, \left(\mathcal{S}^E(v_h, \, v_h) \right)^{1/2} 
  \lesssim |v_2|_{1,E} \left(\mathcal{S}^E(v_h, \, v_h) \right)^{1/2} \,. 
\end{split}
\end{equation}
The bound of the second integral in \eqref{eq:4.3} follows by \eqref{eq:inv2} and \eqref{eq:lorenzo}:
\begin{equation}
\label{eq:4.5}
\begin{split}
\int_E \Delta v_2 \, v_1 \, {\rm d}E  
& \leq \|\Delta v_2\|_{0, E} \, \|v_1\|_{0, E} 
  \lesssim  h_E^{-1} \, |v_2|_{1, E} \, \|v_1\|_{0, E} 
  \lesssim |v_2|_{1, E} \, \left(\mathcal{S}^E(v_h, \, v_h) \right)^{1/2}\,.
\end{split}
\end{equation}
Therefore by \eqref{eq:4.3}, \eqref{eq:4.4}  and \eqref{eq:4.5}
\begin{equation}
\label{eq:4.6}
|v_2|^2_{1,E} \lesssim  \mathcal{S}^E(v_h, \, v_h) \,.
\end{equation}
The thesis follows by collecting \eqref{eq:4.2} and \eqref{eq:4.6} in \eqref{eq:pitagora}.
\end{proof}
We can state the main result of the section. 

\begin{proposition}
\label{prp:stability}
Let any $\epsilon \in (0, 1/2)$. There exist two positive uniform constants $\alpha_*$ and $\alpha^*$ such that for any element $E \in \Omega_h$ it holds that
\begin{gather}
\label{eq:stability1}
 a_h^E (v, \, v) \geq \alpha_* \, a^E (v, \, v) \,,  \\
 \label{eq:stability2}
 a_h^E (v + p_k, \, v + p_k)  \leq \alpha^* \, 
 \left( 
 | v + p_k|_{1,E}^2 +
 |(I -\Pi_k^{\nabla, E}) v |_{1,E}^2 +
 h^{2 \epsilon} |(I -\Pi_k^{\nabla, E}) v |_{1+\epsilon,E}^2  
 \right) \,, 
\end{gather}
for all $v \in V_h^E$ and  $p_k \in \Pk_k(E)$.
\end{proposition}
\begin{proof}
First of all we introduce a useful scaled Poincar{\'e} inequality.
Let $w \in H^1(E)$ and let 
\begin{equation}
\label{eq:c}
\Pi_{0}^{0, \partial E} w:=  \frac{1}{|\partial E|} \int_{\partial E} w \, {\rm d}s \in \Pk_0(E) \,,
\end{equation}
then it holds that \cite{brenner2003poincare}
\begin{equation}
\label{eq:cc}
\|w - \Pi_{0}^{0, \partial E} w \|_{0, E} \leq C \, h_E \, |w|_{1, E} \, .
\end{equation}
Let us analyse the bound \eqref{eq:stability1}.
Let $\widetilde{v} := v - \Pi_{0}^{0, \partial E} v$, then Lemma \ref{lm:coercivity} 
and some simple algebra yield
\begin{equation}
\label{eq:loreplus}
\begin{split}
a^E(v, \, v) = a^E(\widetilde{v}, \, \widetilde{v})
& \lesssim  \mathcal{S}^{E} (\widetilde{v}, \, \widetilde{v})  
\\
& \lesssim 
\mathcal{S}^{E} (\Pi_k^{\nabla, E}\widetilde{v}, \, \Pi_k^{\nabla, E}\widetilde{v}) + 
\mathcal{S}^{E} ((I -\Pi_k^{\nabla, E} )\widetilde{v}, \, (I -\Pi_k^{\nabla, E} )\widetilde{v}) \,.
\end{split}
\end{equation}
By definition \eqref{eq:Pn_k^E}, it  clearly holds that $(I -\Pi_k^{\nabla, E} )\widetilde{v} = (I -\Pi_k^{\nabla, E} )v$.
From Lemma \ref{lm:continuity}, applied to the first term in the right hand side of \eqref{eq:loreplus},
 we then infer 
\begin{equation}
\label{eq:loreplus0}
\begin{split}
a^E(v, \, v)
\lesssim
h_E^{-2} \, \|\Pi_k^{\nabla, E}\widetilde{v}\|_{0, E}^2 + h_E^{2 \epsilon} \, |\Pi_k^{\nabla, E}\widetilde{v}|_{1+\epsilon, E}^2 +
\mathcal{S}^{E} ((I -\Pi_k^{\nabla, E} )v, \, (I -\Pi_k^{\nabla, E} )v)  \,. 
\end{split}
\end{equation}
We observe now that, by definition \eqref{eq:Pn_k^E}, it follows that
\[
\Pi_0^{0, \partial E} \left(\Pi_k^{\nabla, E}\widetilde{v} \right) =
\Pi_0^{0, \partial E} \widetilde{v} =
\Pi_0^{0, \partial E} \left( v - \Pi_0^{0, \partial E} v\right) = 0 
\]
therefore from \eqref{eq:cc} we obtain
\begin{equation}
\label{eq:loreplus1}
h_E^{-2} \, \|\Pi_k^{\nabla, E}\widetilde{v}\|_{0, E}^2 \lesssim  
 |\Pi_k^{\nabla, E}\widetilde{v}|_{1, E}^2 \,.
\end{equation} 
Moreover a standard polynomial inverse estimate on star-shaped polygons yields
\begin{equation}
\label{eq:loreplus2}
h_E^{2 \epsilon} \, |\Pi_k^{\nabla, E}\widetilde{v}|_{1+\epsilon, E}^2 \lesssim  
 |\Pi_k^{\nabla, E}\widetilde{v}|_{1, E}^2  \,.
\end{equation} 
Now \eqref{eq:stability1} follows by inserting  the bounds \eqref{eq:loreplus1} and \eqref{eq:loreplus2} in \eqref{eq:loreplus0} and by observing that $|\Pi_k^{\nabla, E}\widetilde{v}|_{1, E} = |\Pi_k^{\nabla, E}v|_{1, E}$.

For what concerns the bound  \eqref{eq:stability2}, 
since $(I -\Pi_k^{\nabla, E})p_k =0$,
by the continuity of $\Pi_k^{\nabla, E}$ with respect to the $H^1$-seminorm, it follows that
\begin{equation}
\label{eq:stability2a}
\begin{split}
a_h^E(v+ p_k, \, v+ p_k) 
&= 
|\Pi_k^{\nabla,E} (v + p_k)|^2_{1,E} +
\mathcal{S}^E((I -\Pi_k^{\nabla, E} )v, \, (I -\Pi_k^{\nabla, E} )v) 
\\
&\leq 
|v + p_k|^2_{1,E} +
\mathcal{S}^E((I -\Pi_k^{\nabla, E} )v, \, (I -\Pi_k^{\nabla, E} )v) \,.
\end{split}
\end{equation}
We remark that definition \eqref{eq:Pn_k^E} implies that $\Pi_0^{0, \partial} (I -\Pi_k^{\nabla, E} )v = 0$, therefore from \eqref{eq:stability2a}, Lemma \ref{lm:continuity}  and \eqref{eq:cc} we infer
\[
a_h^E(v+ p_k, \, v+ p_k) 
\lesssim
|v + p_k|^2_{1,E} +
|(I -\Pi_k^{\nabla, E} )v|_{1, E}^2 +
h_E^{2 \epsilon} |(I -\Pi_k^{\nabla, E} )v|_{1 + \epsilon, E}^2  \,.
\]
\end{proof}


\subsection{Convergence analysis}
\label{sub:convergence}

%
%
We have proved the stability of the method (Proposition \ref{prp:stability}) 
and the approximation properties of the space $V_h$ (Theorem \ref{thm:interpolante}).
We are now ready to prove the following optimal convergence result.
\begin{theorem}
\label{thm:convergence}
Under the assumptions $\mathbf{(A0)}$, $\mathbf{(A1)}$ and $\mathbf{(A2)}$, let $u$ be the solution of Problem \eqref{eq:continua} and $u_h$ be the solution of virtual Problem \eqref{eq:virtual}. Assume moreover $u \in H^{s}$ and $f \in H^{s-2}$, with $\frac{3}{2} \leq s \leq k +1$. Then 
\begin{equation}
\label{eq:error}
|u - u_h|_1 \leq C \, h^{s-1} (\|u\|_{s} + |f|_{s-2})
\end{equation}
where the constant $C$ depends on the degree $k$, the parametrization $\gamma$ and the shape regularity constant $\rho$.
\end{theorem}

\begin{proof}
Let $u_{\rm I}$ be the interpolant of $u$ in the sense of Theorem \ref{thm:interpolante} and $u_{\pi}$ be the polynomial approximation of $u$ given by Lemma \ref{lm:scott}. 
Let us set $e_h := u_h - u_{\rm I}$, then, following standard steps in VEM analysis \cite{volley}, we have
\begin{equation}
\label{eq:convergence1}
\begin{split}
\alpha_* \, |e_h|_1^2 
& \leq  a_h(e_h, \, e_h) 
=a_h(u_h - u_{\rm I}, \, e_h) \\
& =    (f_h - f, \, e_h) + 
\sum_{E \in \Omega_h} \left( a_h^E(u_{\pi} - u_{\rm I}, \, e_h) + a^E(u - u_{\pi}, \, e_h) \right) \\
&\leq C\,  \, h^{s-1} |f|_{s-2} \, |e_h|_1 + 
\sum_{E \in \Omega_h} a_h^E(u_{\pi} - u_{\rm I}, \, e_h) +
C  \,h^{s-1}\, |u|_{s} \, |e_h|_1
 \,.
\\
\end{split}
\end{equation}
From the Cauchy-Schwarz inequality we infer
\[
a_h^E(u_{\pi} - u_{\rm I}, \, e_h)  \leq 
\frac{1}{2} \, a_h^E(u_{\pi} - u_{\rm I}, \, u_{\pi} - u_{\rm I})    +
\frac{1}{2} \, a_h^E(e_h, \, e_h) \,.
\]
Inserting the previous inequality in \eqref{eq:convergence1}, we get
\begin{equation}
\label{eq:convergence2}
\begin{split}
\frac{\alpha_*}{2} \, |e_h|_1^2 
&\lesssim \, h^{s-1}  \left( |f|_{s-2} + |u|_{s}  \right) \, |e_h|_1 +
\sum_{E \in \Omega_h} a_h^E(u_{\pi} - u_{\rm I}, \, u_{\pi} - u_{\rm I}) \,.
\end{split}
\end{equation}
Now fix any one $\epsilon \in (0, 1/2)$. 
From \eqref{eq:stability2}, with $v_h =- u_{\rm I}$ and $p_k = u_{\pi}$, we have
\begin{equation}
\label{eq:convergence3}
\begin{split}
a_h^E(u_{\pi} - u_{\rm I}, \, u_{\pi} - u_{\rm I}) &\lesssim
|u_{\pi} - u_{\rm I}|^2_{1,E} +
|(I -\Pi_k^{\nabla, E} ) u_{\rm I}|_{1, E}^2 +
h_E^{2 \epsilon} |(I -\Pi_k^{\nabla, E}) u_{\rm I}|_{1 + \epsilon, E}^2 
\\
& =: \mu_1^E + \mu_2^E + \mu_3^E \,.
\end{split}
\end{equation}
The first term, using Lemma \ref{lm:scott},  is estimated as follows
\begin{equation}
\label{eq:mu1}
\mu_1^E \lesssim |u_{\pi} - u|^2_{1,E} + |u - u_{\rm I}|^2_{1,E} \lesssim h^{2s-2} \, |u|_{s, E}^2 + |u - u_{\rm I}|^2_{1,E} \,.
\end{equation}
Concerning the second term, by the continuity of the $\Pi_k^{\nabla, E}$ projection and Lemma \ref{lm:scott}   we have
\begin{equation}
\label{eq:mu2}
\begin{split}
\mu_2^E &\lesssim |(I -\Pi_k^{\nabla, E} ) ( u - u_{\rm I})|_{1, E}^2 + |(I -\Pi_k^{\nabla, E} )  u|_{1, E}^2 \lesssim |u - u_{\rm I}|_{1,E}^2 + h^{2s-2} \, |u|_{s, E}^2 \,.
\end{split}
\end{equation}
Finally the last term is handled using Lemma \ref{lm:scott} and standard polynomial inverse estimates on star-shaped domains
\begin{equation}
\label{eq:mu3}
\begin{split}
\mu_3^E &\lesssim 
h^{2 \epsilon} \, | u - u_{\rm I}|_{1 + \epsilon, E}^2 +
h^{2 \epsilon} \,|(I -\Pi_k^{\nabla, E} )  u|_{1+ \epsilon, E}^2 +
h^{2 \epsilon} \,|\Pi_k^{\nabla, E} (u - u_{\rm I})|_{1+ \epsilon, E}^2 \\
&\lesssim 
h^{2 \epsilon} \,  | u - u_{\rm I}|_{1 + \epsilon, E}^2 +
h^{2s-2} \,|u|_{s,E}^2 +  |u - u_{\rm I}|_{1, E}^2 \,.
\end{split}
\end{equation}
Collecting \eqref{eq:mu1} \eqref{eq:mu2} and \eqref{eq:mu3} in \eqref{eq:convergence3} we obtain
\[
a_h^E(u_{\pi} - u_{\rm I}, \, u_{\pi} - u_{\rm I}) \lesssim 
 | u - u_{\rm I}|_{1, E}^2 + 
h^{2 \epsilon} \,  | u - u_{\rm I}|_{1 + \epsilon, E}^2 +
h^{2 s -2 } \, |u|_{s, E}^2 
\]
so that from Theorem \ref{thm:interpolante} we infer
\begin{equation}
\label{eq:convergence4}
\sum_{E \in \Omega_h} a_h^E(u_{\pi} - u_{\rm I}, \, u_{\pi} - u_{\rm I}) \lesssim  h^{2 s -2 } \, \|u\|_s^2 \,.
\end{equation}
The thesis now easily follows from \eqref{eq:convergence2} and \eqref{eq:convergence4}.
\end{proof}

Regularity results on Lipschitz domains (see \cite{Grisvard:1992}) guarantee
that the solution $u$ of Problem \eqref{eq:continua} is in $H^{s}$ with $s \geq \frac{3}{2}$. 
We furthermore note that the condition $s \geq \frac{3}{2}$ in the statement of Theorem \ref{thm:convergence} could be reduced to $s>1$ by modifying Theorem \ref{thm:interpolante} in accordance with Remark \ref{rm:piddu}.

\section{Integration on curvilinear polygons}
\label{sec:4}
Let $E \in \Omega_h$ a curved element. The boundary $\partial E$ is described counterclockwise by the sequence of edges $e_1, e_2, \dots, e_{N_E}$ where, in accordance with the notation of Section \ref{sub:2.2} we assume that $e_1$ is a curved edge and $e_2, \dots, e_{N_E}$ are straight segments (see Figure \ref{fig:vianello0}).
\begin{figure}[!h]
\center{
\begin{overpic}[scale=0.30]{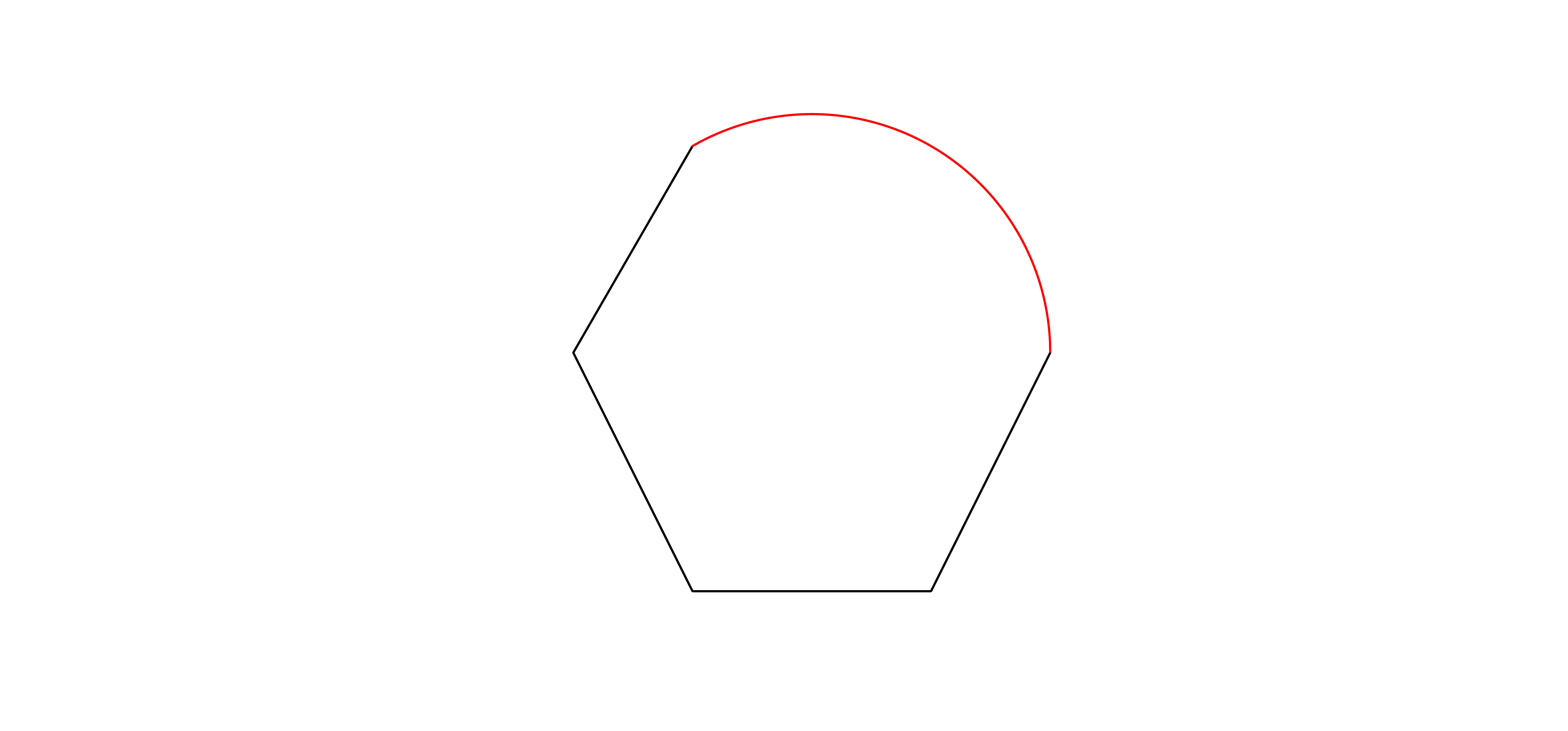}
\put (50,25) {{$E$}}
\put (65,35) {{$e_1$}}
\put (37,31) {{$e_2$}}
\put (37,15) {{$e_3$}}
\put (50,6) {{$e_4$}}
\put (65,15) {{$e_5$}}
\put (37,38) {{$(x_2, y_2)$}}
\put (29,24) {{$(y_3, y_3)$}}
\put (37,7) {{$(x_4, y_4)$}}
\put (60,7) {{$(x_5, y_5)$}}
\put (68,24) {{$(x_1, y_1)$}}
\end{overpic}
\vspace*{-1cm}
\caption{Curved domain $E$.}
\label{fig:vianello0}
}
\end{figure}
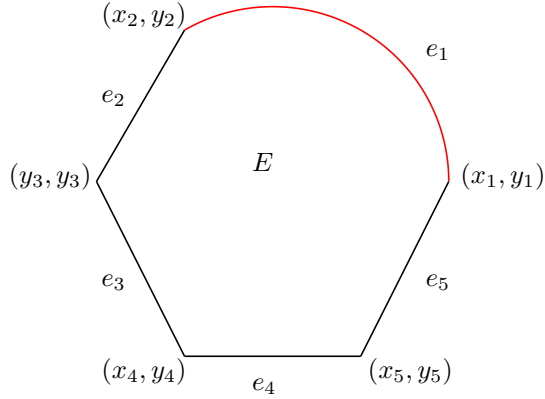
Following the guidelines of \cite{VEM-hitchhikers}, 
in order to compute the approximated bilinear form $a_h(\cdot, \, \cdot)$ and the discrete right hand side $(f_h, \, \cdot)$ we need to compute (or to approximate) for all $E \in \Omega_h$ the following integrals
\begin{align}
\label{eq:integral1}
&\int_{\partial E} v_h \, p_{k-1} \, {\rm d}s  &\qquad &\text{for all $v_h \in V_h^E$,  $p_{k-1} \in \Pk_{k-1}(E)$,} \\
\label{eq:integral2}
&\int_{E} q_{k-1} \, p_{k-1} \, {\rm d}E  &\qquad &\text{for all $q_{k-1}$, $p_{k-1} \in \Pk_{k-1}(E)$,} \\
\label{eq:integral3}
&\int_{E} f \, p_{k-2} \, {\rm d}E &\qquad &\text{for all $p_{k-2} \in \Pk_{k-2}(E)$.} 
\end{align} 
The goal of this section is to show the practical computations of the previous integrals. The main ideas we are going to develop are
\begin{itemize}
\item the boundary integrals are computed only using the DoFs (i.e. values at the Gauss-Lobatto nodes) and the parametrization $\gamma$,
\item the integrals on the element are computed avoiding partitions of the curved polygonal
domain $E$ into triangles or quadrangles.
\end{itemize}

\subsection{Integration on curved edges}
\label{sub:4.0}
The integral \eqref{eq:integral1} can be split edge by edge obtaining
\begin{equation}
\int_{\partial E} v_h \, p_{k-1} \, {\rm d}s = 
\int_{e_1} v_h \, p_{k-1} \, {\rm d}s + 
\sum_{i=2}^{N_E} \int_{e_i} v_h \, p_{k-1} \, {\rm d}s \,.
\end{equation}
We recall that, by definition \eqref{eq:virtualpace}, the virtual functions restricted to each straight edge are polynomials of degree $k$, therefore the integral on straight edges $e_i$ can be computed in the standard way \cite{VEM-hitchhikers} using the DoFs $\boldsymbol{D^{e_i}}$ 
(observing that the $(k+1)$ Gauss-Lobatto quadrature rule  is exact for polynomials 
up to degree $2k-1$).

Instead, on the curved edge $e_1$ the virtual functions are polynomial of degree $k$ with respect the parametrization $\gamma \colon I_{e_1} \to e_1$, i.e.
$v_{h_{|e_1}} = \widehat{q_k} \circ \gamma$ with $\widehat{q_k} \in \Pk_k(I_{e_1})$.
We approximate the integral on the curved edge using the same machinery of the straight case: let  
$$
\{\tau_j\}_{j=1}^{k+1} \qquad \text{and} \qquad \{\omega_j\}_{j=1}^{k+1},
$$ 
be the nodes and weights of the 
Gauss-Lobatto quadrature formula of degree of exactness $2k - 1$ on $I_{e_1}$. 
Recalling \eqref{eq:correspondences}, 
we set
\begin{multline}
\int_{e_1} v_h \, p_{k-1}  \, {\rm d}s = 
\int_{e_1} \widehat{q_k \, p_{k-1}} \circ \gamma^{-1} \, {\rm d}s  = \\
=
\int_{I_e} \widehat{q_k \, p_{k-1}} \|\gamma'\| \, {\rm d}t 
\approx \sum_{j = 1}^{k+1} \widehat{q_k \, p_{k-1}}(\tau_j) \, \|\gamma'(\tau_j)\| \, \omega_j = \\
=v_h (\gamma (\tau_j)) \, p_{k-1} (\gamma (\tau_j)) \, \|\gamma'(\tau_j)\| \, \omega_j 
= \sum_{j = 1}^{k+1} \boldsymbol{D}^{\boldsymbol{e_1}}_j(v_h) \, p_{k-1} (\gamma (\tau_j)) \,
\|\gamma'(\tau_j)\| \, \omega_j \,.
\end{multline}
We remark that the integration rule above is a trivial extension of the straight case, in particular if $e_1$ is a straight side
we recover the standard rule.

\subsection{Integration on polygons}
\label{sub:4.1}

In the present section we review the basic tools for the construction
of a quadrature rule over (possibly non convex) polygons
that is exact for polynomials of degree $2M$ 
and uses $O(M^2)$ nodes.
The quadrature formula was introduced in \cite{vianello1, vianello2}.
Let $E \subset \R^2$ be a bounded and simply connected
polygon and let $\mathcal{R}$ a suitable rectangle containing $E$.
We consider the problem of constructing a quadrature formula (with nodes in $\mathcal{R}$)
which is exact for all polynomials of degree at most $2M$.
%
\begin{proposition}
\label{prp:vianello}
Let the boundary $\partial E$ be described counterclockwise by the sequence of vertexes
\[
\begin{gathered}
V_i := (x_i, \, y_i) \qquad \text{for $i=1, \dots, N_E$} \\
\partial E =  \cup_{i=1}^{N_E} e_i := \cup_{i=1}^{N_E}[V_i, \, V_{i+1}], \qquad \text{with $V_{N_E + 1} = V_1$} 
\end{gathered}
\]
and let
\[ 
\begin{gathered}
\{t_j\}_{j=1}^{M+1}  \qquad \text{and} \qquad \{\nu_j\}_{j=1}^{M+1},
\end{gathered}
\]	
be the nodes and weights of the Gauss-Legendre quadrature formula of degree of exactness 
$2M+1$ on $[-1, 1]$.
Let $g \in C(\mathcal{R})$ where $\mathcal{R}$ is a rectangle containing $E$.
Then the following quadrature formula is exact over $E$ for all polynomials of degree at most $2M$
\[
\iint_{E} g(x, \, y) \, {\rm d}x  \, {\rm d}y 
\approx
\sum_{e_i \in \partial E}
\sum_{j = 1}^{M+1}
\sum_{m = 1}^{M+1}
g(x_i(t_{j,m}), \, y_i(t_{j})) \, w_{ijm}
\]
where
\[
\begin{aligned}
x_i(t_{j,m}) &:= \frac{x_i(t_j) - \alpha}{2} t_m + \frac{x_i(t_j) + \alpha}{2},
\qquad \text{with} \qquad
x_i(t_j) := \frac{x_{i+1} - x_i}{2} t_j + \frac{x_{i+1} + x_i}{2} \,, \\
y_i(t_j) &:= \frac{y_{i+1} - y_i}{2} t_j + \frac{y_{i+1} + y_i}{2} \,, \\
w_{ijm} &:= \frac{y_{i+1} - y_i}{4} (x_i(t_j) - \alpha) \, \nu_j \, \nu_m \,,
\end{aligned}
\]
and $\alpha$ is a convex combination of the coordinate $x_i$'s.
\end{proposition}
\begin{proof}
We here review the main idea of the proof (see \cite{vianello1}, Theorem 2.1) since it is useful in order to understand the extension on curved polygons of Section \ref{sub:4.2}.
Let $p(x, \, y)$ be a polynomial of degree $2M$.
For a suitable convex combination $\alpha$ of the $x_i$'s coordinates, let  
\begin{equation}
\label{eq:vianelloo}
\mathcal{P}(x, \, y) := \int_{\alpha}^x p(t, \, y) \,{\rm d}t
\end{equation}
be an $x$-primitive of the function $p(x, \, y)$, so that by Green's formula it holds 
\begin{equation}
\label{eq:vianello0}
\iint_{E} p(x, \, y) \, {\rm d}x  \, {\rm d}y =
\int_{\partial E} \mathcal{P}(x, \, y) \, {\rm d}y = 
\sum_{e_i \in \partial E}
\int_{e_i}  \mathcal{P}(x, \, y) \, {\rm d}y \,.
\end{equation}
This transforms a $2$-dimensional integration problem into a $1$-dimensional problem.
Now, we observe that being $p(x,\, y)$ a polynomial of degree at most $2M$, 
$\mathcal{P}(x, \, y)$ is a polynomial of degree at most $2M+1$, so that using the Gauss-Legendre quadrature rule of degree $2M+1$ we obtain
\begin{equation}
\label{eq:vianello1}
\int_{e_i}  \mathcal{P}(x, \, y) \, {\rm d}y =
\sum_{j = 1}^{M+1} \frac{y_{i+1} - y_i}{2} \mathcal{P}(x_i(t_j), \, y_i(t_j)) \, \nu_j \,.
\end{equation}
We notice now that $\mathcal{P}(x_i(t_j), \, y_i(t_j))$ is unknown, but by definition \eqref{eq:vianelloo}, it can be calculated by a Gauss-Legendre quadrature of degree $2M+1$ on the interval $[\alpha, \, x_i(t_j)]$:
\begin{equation}
\label{eq:vianello2}
\mathcal{P}(x_i(t_j), \, y_i(t_j)) = 
\int_{\alpha}^{x_i(t_j)} p(t, \, y) \,{\rm d}t
=
\sum_{m = 1}^{M+1} \frac{x_i(t_j) - \alpha}{2} p(x_i(t_{j,m}), \, y_i(t_j)) \, \nu_m \,.
\end{equation}
Collecting \eqref{eq:vianello1} and \eqref{eq:vianello2} in \eqref{eq:vianello0} we get the thesis.
\end{proof}

A careful inspection of Proposition \ref{prp:vianello} shows that the overall number of quadrature nodes is bounded by $(M+1)^2\,N_E$. 
Finer estimate of the total amount of quadrature points is given in \cite{vianello1}. Moreover in \cite{vianello4, vianellocodici} the same authors present a compression procedure to reduce the number of nodes in the numerical quadrature. 
A careful inspection of \cite{vianello4} shows that in general the expected compression ratio is
\[
\frac{\text{\texttt{n nodes with compression}}}
{\text{\texttt{n nodes without compression}}}
= \frac{2 M + 1}{N_E \, M}\,.
\]
The suggested way to extract the nodes
is to solve a Non Negative Least Squares problem which ensures positivity of the weights. The compression procedure turns out to be very convenient if the number of edges $N_E$ of the polygon is large.

\begin{remark}
A possible drawback of the quadrature formula
is that the quadrature nodes may
fall  outside of the polygon $E$. This
is the reason why in Proposition \ref{prp:vianello} $g$ is assumed to be continuous and computable also in a rectangle $\mathcal{R}$ containing $E$.
However, this effect can be eliminated in certain cases (e.g. $E$ convex) with a clever selection of the ``integration line'' $l$ (in the proof of Proposition \ref{prp:vianello} $l: x= \alpha$), for instance
considering $l$  as the line containing  the diameter  of the polygon and then using a change of variables.
See Remark 2.4 in \cite{vianello1} for a more detailed discussion.
\end{remark}

\begin{remark}
The quadrature formula in Proposition \ref{prp:vianello} avoids explicit a priori partition of the polygonal
domain $E$ into triangles or quadrangles, since
Green's formula needs only the boundary as a counterclockwise sequence of
vertexes.
\end{remark}

\subsection{Gauss quadrature over curvilinear polygons}
\label{sub:4.2}

Let $E \in \Omega_h$ be a curved polygon as in Figure \ref{fig:vianello0}. The main goal of the section is to define suitable approximation of the integrals \eqref{eq:integral2} and \eqref{eq:integral3}. The idea is to extend the quadrature formula of Proposition \ref{prp:vianello} to curved elements. We observe that by \eqref{eq:integral2} the formula has to be exact (ideally) for polynomials of degree $2k-2$. So that, let $p(x, \, y) \in \Pk_{2k -2}(E)$ and let
\begin{equation}
\label{eq:vianello10}
\mathcal{P}(x, \, y)   = 
\int_{\alpha}^x  p(t, \, y) \, {\rm d}t
\end{equation}
where, using the notation in Figure \ref{fig:vianello0}, we pick
$\alpha := \frac{1}{N_E} \sum_{i=1}^{N_E} x_i$.
Using the same machinery of the proof of Proposition \ref{prp:vianello} we get
\begin{equation}
\label{eq:vianello11}
\iint_E p(x, \, y) \, {\rm d}x\, {\rm d}y =
\int_{\partial E} \mathcal{P}(x, \, y)\, {\rm d}y = 
\int_{e_1} \mathcal{P}(x, \, y)\, {\rm d}y + 
\sum_{i=2}^{N_E} \int_{e_i} \mathcal{P}(x, y)\, {\rm d}y \,.
\end{equation}
On the straight sides we can proceed as before setting $M= k-1$.
Let us analyse the integral involving the curved edge $e_1$.
Let $\gamma \colon I_{e_1} \to e_1$ described by the components $\gamma(t) := (\gamma_1(t), \, {\gamma_2}(t))$ and let 
$(\widetilde{x_1}(t_j), \, \widetilde{y_1}(t_j))$ for $j=1, \dots k$,
be the images through $\gamma$ of the $k$ nodes of the  Gauss-Legendre quadrature on $I_{e_1}$ of degree $2k-1$ (see Figure \ref{fig:vianello1}). Then we set
\begin{equation}
\label{eq:vianello12}
\begin{split}
\int_{e_1} \mathcal{P}(x, \, y)\, {\rm d}y &=
\int_{I_e} \mathcal{P}(\gamma_1(s), \, \gamma_2(s))\, \gamma_2'(s) \, {\rm d}s 
\approx
\frac{|I_{e_1} |}{2} \, \sum_{j=1}^{k} \mathcal{P} \left(\widetilde{x_1}(t_j), \widetilde{y_1}(t_j) \right) \,  \gamma_2'(t_j)  \,\nu_j \,
\end{split} 
\end{equation}
where $\{ \nu_j \}_{j=1}^{k}$ are the weights of the Gauss-Legendre formula on $[-1, 1]$.
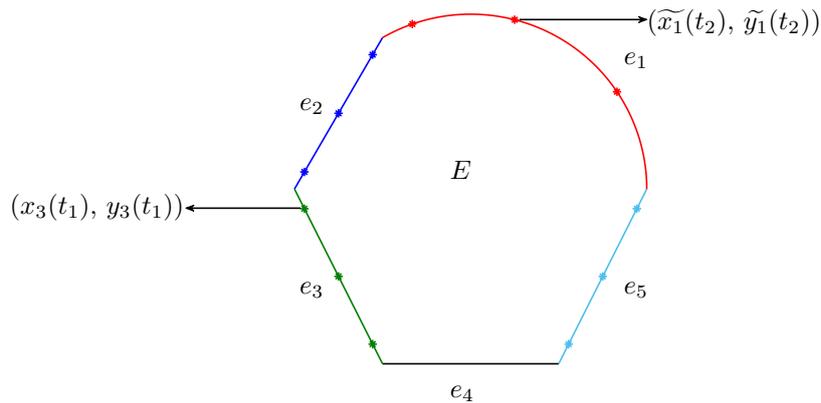
\begin{figure}[!h]
\center{
\begin{overpic}[scale=0.30]{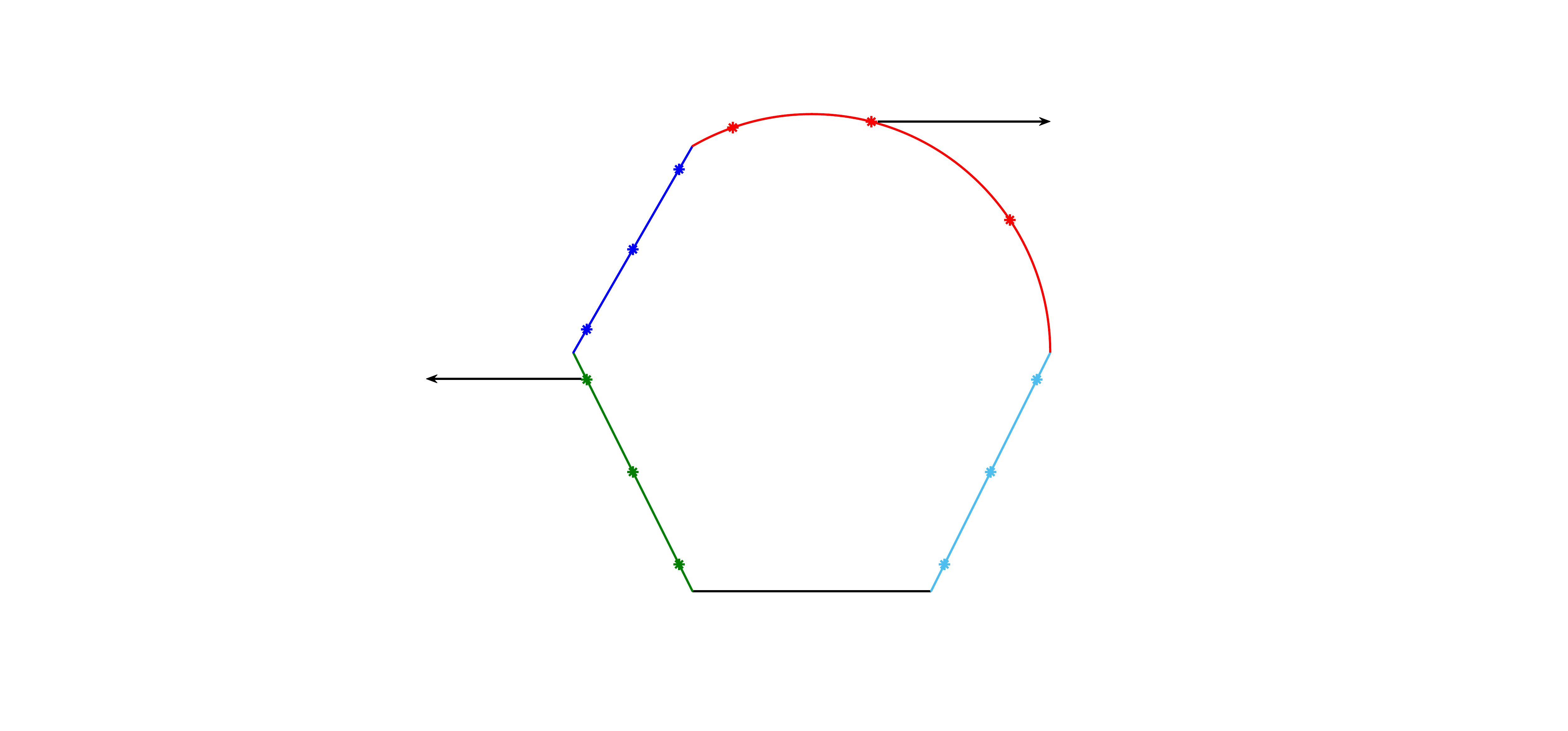}
\put (50,25) {{$E$}}
\put (65,35) {{$e_1$}}
\put (37,31) {{$e_2$}}
\put (37,15) {{$e_3$}}
\put (50,6) {{$e_4$}}
\put (65,15) {{$e_5$}}
\put (67,38) {{$\left(\widetilde{x_1}(t_2), \, \widetilde{y_1}(t_2)\right)$}}
\put (12,22) {{$\left(x_3(t_1), \, y_3(t_1)\right)$}}
\end{overpic}
\vspace*{-1cm}
\caption{Example of distribution of Gauss-Legendre points on the boundary ($k=3$) for a curved domain $E$.}
\label{fig:vianello1}
}
\end{figure}
Using again the Gauss-Legendre quadrature formula  on the interval $[\alpha, \, \widetilde{x_1}(t_j)]$, by \eqref{eq:vianello10} (see Figure \ref{fig:vianello2}), we obtain:
\begin{equation}
\label{eq:vianello13}
\begin{split}
\mathcal{P}  \left(\widetilde{x_1}(t_j), \widetilde{y_1}(t_j) \right) 
=
\frac{\widetilde{x_1}(t_j) - \alpha}{2} \,
\sum_{m=1}^{k+1} \,  p\left(\widetilde{x_1}(t_{j,m}), \,\widetilde{y_1}(t_j) \right) \, \nu_m\, .
\end{split}
\end{equation}
%
%
%
%
%
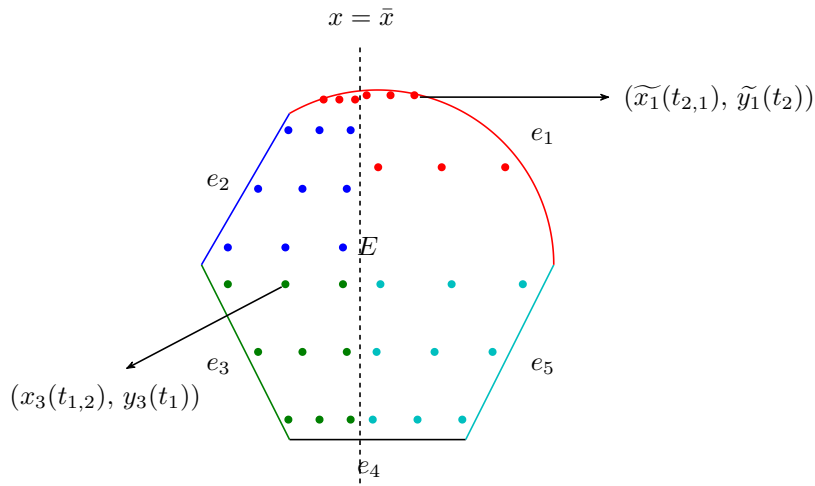
\begin{figure}[!h]
\center{
\begin{overpic}[scale=0.30]{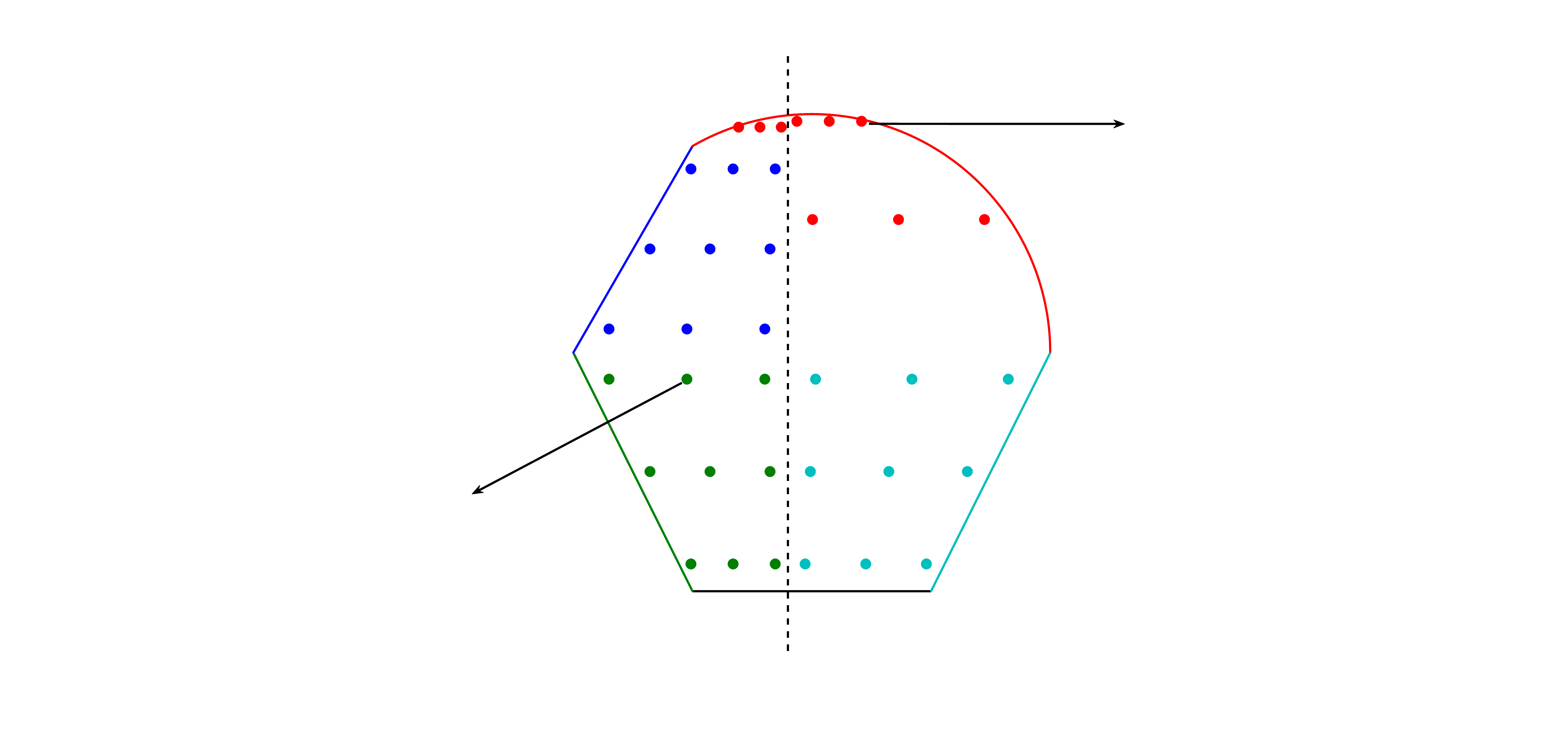}
\put (47.5,45) {{$x = \bar{x}$}}
\put (50,25) {{$E$}}
\put (65,35) {{$e_1$}}
\put (37,31) {{$e_2$}}
\put (37,15) {{$e_3$}}
\put (50,6) {{$e_4$}}
\put (65,15) {{$e_5$}}
\put (73,38) {{$\left(\widetilde{x_1}(t_{2,1}), \, \widetilde{y_1}(t_2)\right)$}}
\put (20,12) {{$\left(x_3(t_{1,2}), \, y_3(t_1)\right)$}}
\end{overpic}
\vspace*{-1cm}
\caption{Example of distribution of quadrature points ($k=3$) for a curved domain $E$.}
\label{fig:vianello2}
}
\end{figure}
Collecting \eqref{eq:vianello12} and \eqref{eq:vianello13} in \eqref{eq:vianello11}, and in the light of the notation of Proposition \ref{prp:vianello}, we obtain  the following quadrature formula on the curved polygon $E$.
Let $g \in C(\mathcal{R})$ where $\mathcal{R}$ is a rectangle containing $E$.
Then the following quadrature formula extends to curved polygon $E$ the rule in Proposition \ref{prp:vianello}
\begin{multline}
\label{eq:vianellocurvo}
\iint_{E} g(x, \, y) \, {\rm d}x  \, {\rm d}y 
\approx
\frac{I_{e_1}}{4} \sum_{j,m = 1}^{k}
g(\widetilde{x_1}(t_{j,m}), \, \widetilde{y_1}(t_{j})) \, \gamma_2'(t_j) 
\, (\widetilde{x_1}(t_j) - \alpha) \,
\nu_j \, \nu_m\,
+ \\
+
\sum_{i=2}^{N_E}
\sum_{j,m = 1}^{k}
g(x_i(t_{j,m}), \, y_i(t_{j})) \, w_{ijm}
\end{multline}
Finally, as for the straight case, we can apply the compression procedure (see Figure \ref{fig:vianello3}).
\begin{figure}[!h]
\center{
\begin{overpic}[scale=0.30]{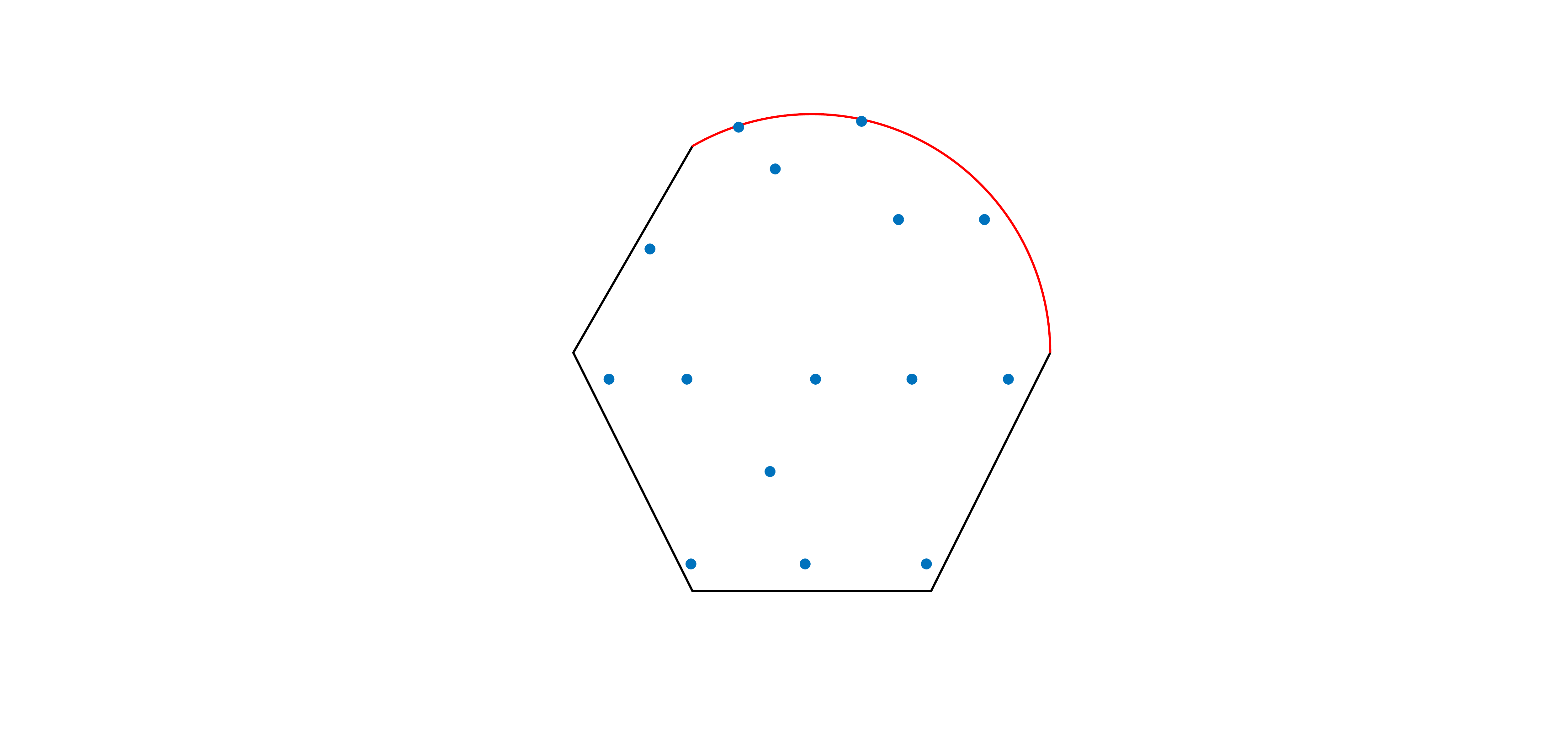}
\put (50,25) {{$E$}}
\put (65,35) {{$e_1$}}
\put (37,31) {{$e_2$}}
\put (37,15) {{$e_3$}}
\put (50,6) {{$e_4$}}
\put (65,15) {{$e_5$}}
\end{overpic}
\vspace*{-1cm}
\caption{Example of distribution of quadrature points ($k=3$) for a curved domain $E$ after the compression procedure.}
\label{fig:vianello3}
}
\end{figure}
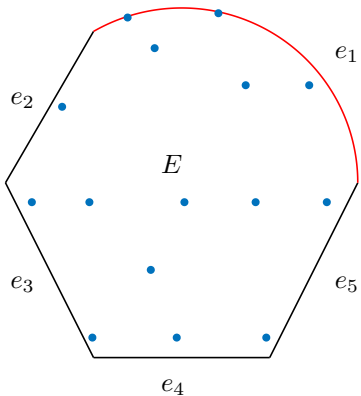
%

\begin{remark}
It is important to note that the accuracy order of the quadrature formulas above is tailored to obtain, in  straight edge case, an accuracy order of $2k-2$. On the other hand, in the presence of curved edges, such formulas will not be exact even for polynomials of the indicated degree. Nevertheless such observation does not seem to affect the numerical results; in the numerical tests of the next section we indeed used the formulas above without any loss in accuracy due to numerical integration. Clearly, if one wants to stay on a safer side, one can opt for integration formulas of the same type but of higher accuracy order, at the expense of some computational cost.

An important related note is the following. From the coding standpoint the present method for curved faces can be easily obtained by a direct modification of the original VEM method for straight edges. Essentially, one keeps the same structure and simply modifies the integration formulas in order to keep into account that the volume (and some edge) integrals are now on curved domains. Thus the code can be built rather easily combining the guidelines in \cite{VEM-hitchhikers} with the integration formulas above. This note is important in order to extend (in a practical way) to curved faces any of the many Virtual Element Methods, for more complex problems, that are in the literature.
\end{remark}

\section{Numerical Tests}
\label{sec:tests}

In this section we present two numerical experiments to test the practical performance of the method. 
In Test \ref{test1} we study the convergence of the proposed method. 
We also check numerically that, as expected, an approximation of the
domain with ``straight VEM'' polygons, leads, for $k \geq 2$, to a sub-optimal rate of convergence.
In Test \ref{test2} we assess the proposed virtual element method for curved domains with curved interface inside the domains.
In order to compute the VEM errors between the exact solution $u_{\rm ex}$ and VEM solution $u_h$, we consider the computable error quantities:
\begin{gather*}
{\rm err}_{H^1} := 
\frac{\left( \sum_{E \in \Omega_h} \left \| \nabla u_{\rm ex} -  \nabla ( \Pi_{k}^{\nabla, E} u_h) \right \|_{0,E}^2 \right)^{1/2}}{|u_{\rm ex}|_1} \,,  \\
{\rm err}_{L^2} := 
\frac{\left( \sum_{E \in \Omega_h} \left \|  u_{\rm ex} -  \Pi_{k}^{\nabla, E} u_h \right \|_{0,E}^2 \right)^{1/2}}{\|u_{\rm ex}\|_0} \,.
\end{gather*}
%
The $H^1$-error confirms the theoretical predictions of Section \ref{sub:convergence}, that is an $O(h^k)$ convergence.  
For what concerns the $L^2$-error, the behaviour is again the expected one $O(h^{k+1})$. Although we did not prove the $L^2$-error estimate in the present paper, this could be derived combining the tools here presented with Section 2.7 in \cite{BeiraodaVeiga-Brezzi-Marini:2013}.

\begin{test}
\label{test1}
We consider the curved domain $\Omega$ described by
\begin{equation}
\label{eq:test1domain}
\Omega := \{ (x, \, y) \quad \text{s.t} \quad 0 \leq x \leq 1, \quad \text{and} \quad g_1(x) \leq y \leq g_2(x) \} \,,
\end{equation}
where 
\[
g_1(x) := \frac{1}{20} \sin(\pi x) 
\qquad \text{and} \qquad
g_2(x) := 1 + \frac{1}{20} \sin(3\pi x) \,
\]
see also Figure \ref{fig:test1domain}. We assume that the curved edges $\Gamma_1$ and $\Gamma_2$ are parametrized with the standard graph parametrization, i.e.
\[
\begin{aligned}
\gamma_1 &\colon [0, \, 1] \to \Gamma_1 
&\qquad \gamma_1(t)& = \left(t, \,  \frac{1}{20} \sin(\pi t)\right) ,
\\
\gamma_2 &\colon [0, \, 1] \to \Gamma_2 
&\qquad \gamma_2(t)& = \left(t, \,  1 + \frac{1}{20} \sin(3 \pi t)\right) ,
\end{aligned}
\]
we notice that both $\Gamma_1$ and $\Gamma_2$ are not parametrized by the arc-length parametrizations.
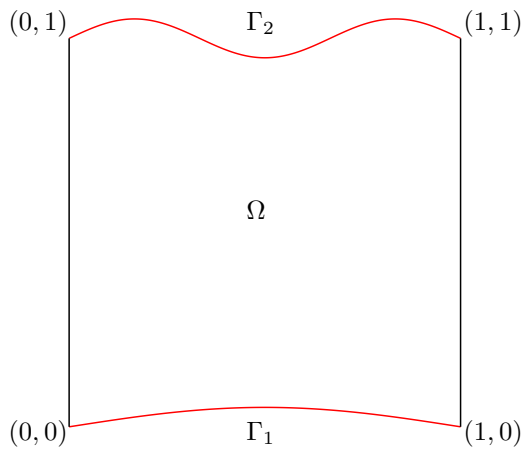
\begin{figure}[!h]
\center{
\begin{overpic}[scale=0.28]{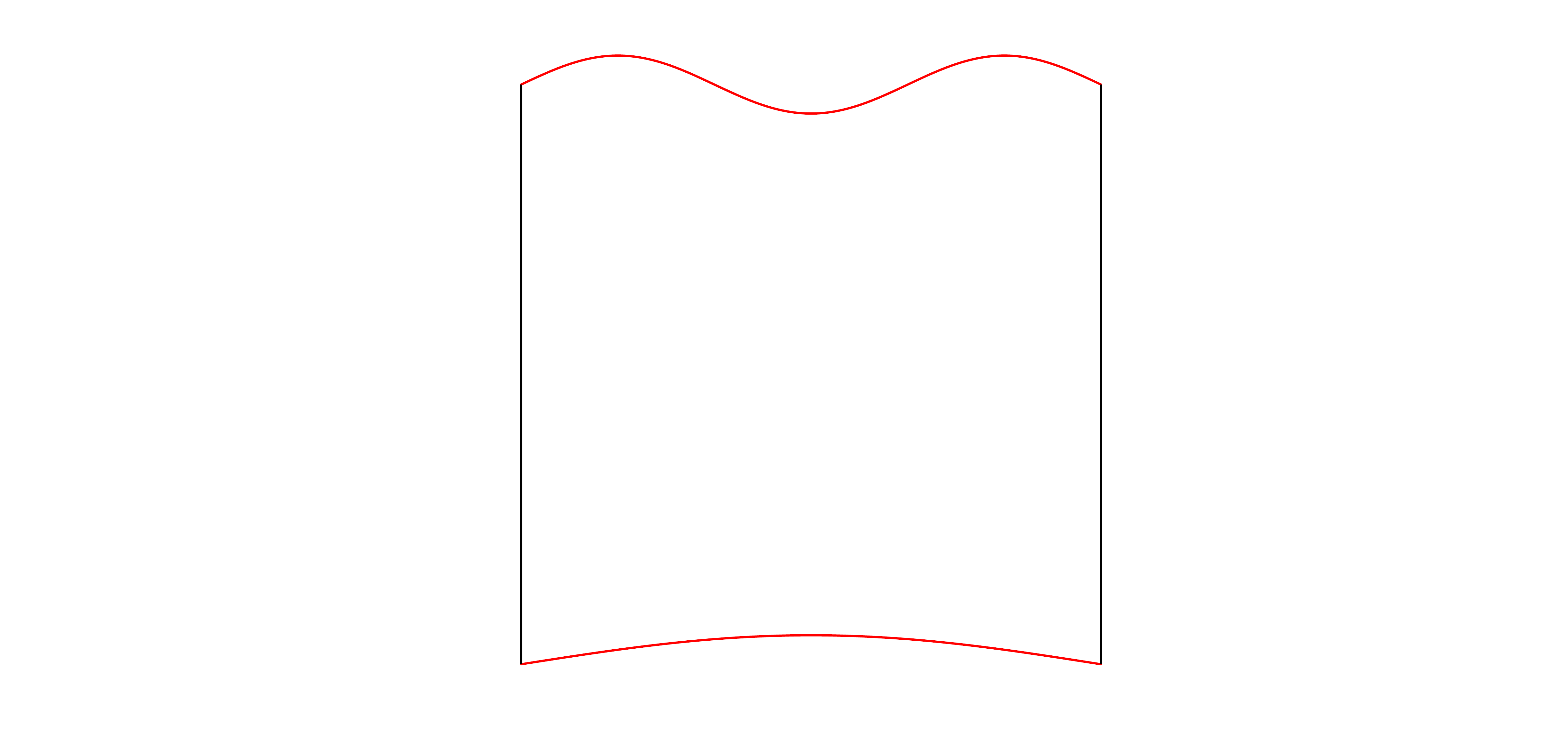}
\put (50,25) {{$\Omega$}}
\put (50,43) {{$\Gamma_2$}}
\put (50,4) {{$\Gamma_1$}}
\put (27.5,4) {{$(0,0)$}}
\put (27.5,43) {{$(0,1)$}}
\put (70.5,4) {{$(1,0)$}}
\put (70.5,43) {{$(1,1)$}}
\end{overpic}
\vspace*{-0.5cm}
\caption{Test \ref{test1}. Domain $\Omega$ described in \eqref{eq:test1domain}.}
\label{fig:test1domain}
}
\end{figure}
In this example we test the elliptic problem \eqref{eq:continua} on the curved domain $\Omega$.  We choose the load term $f$ and the boundary conditions in such a way that the analytical solution is  
\begin{equation}
\label{eq:test1solution}
u_{\rm ex}(x, \, y) = -(y - g_1(x))(y - g_2(x)) (3 + \sin(5x) \sin(7y)).
\end{equation}
In particular, from \eqref{eq:test1solution}, it yields that 
\begin{equation}
\label{eq:test1bc}
{u_{\rm ex}}_{|{\Gamma_1 \cup \Gamma_2}} =0 \,.
\end{equation}
The aim of this test is to check the actual performance of the virtual element method for curved domains, in comparison with the standard virtual element on straight domains obtained by approximating the curved boundary by straight segments.
The domain $\Omega$ is partitioned with two  sequences of polygonal meshes: the Voronoi meshes and uniform quadrilateral meshes. Both sequences of meshes are obtained defining a sequence of Voronoi (resp. uniform) meshes on the square $Q:= [0, \, 1]^2$ and then moving the vertexes  of each polygon by the rule
\[
(x_Q, \, y_Q) \mapsto (x_{\Omega}, \, y_{\Omega}) :=
\left\{
\begin{aligned}
&  \bigl(x_Q, \, (1 - 2g_1(x_Q)) y_Q + g_1(x_Q) \bigr) \qquad &  
\text{if $y_Q \leq \frac{1}{2}$,}\\
& \left(x_Q, \, (2g_2(x_Q) - 1) y_Q + \frac{1}{2} \right) \qquad &  
\text{if $y_Q \geq \frac{1}{2}$,}
\end{aligned}
\right.
\] 
where $(x_Q, \, y_Q)$ and $(x_{\Omega}, \, y_{\Omega})$ denote respectively the mesh nodes on the square domain $Q$ and on the curved domain $\Omega$. The edges on the  curved  boundary consist of an arc of $\Gamma_1$ or $\Gamma_2$.
In Figure \ref{fig:test1meshes} we display an example of the adopted meshes.
\begin{figure}[!h]
\center{
\includegraphics[scale=0.25]{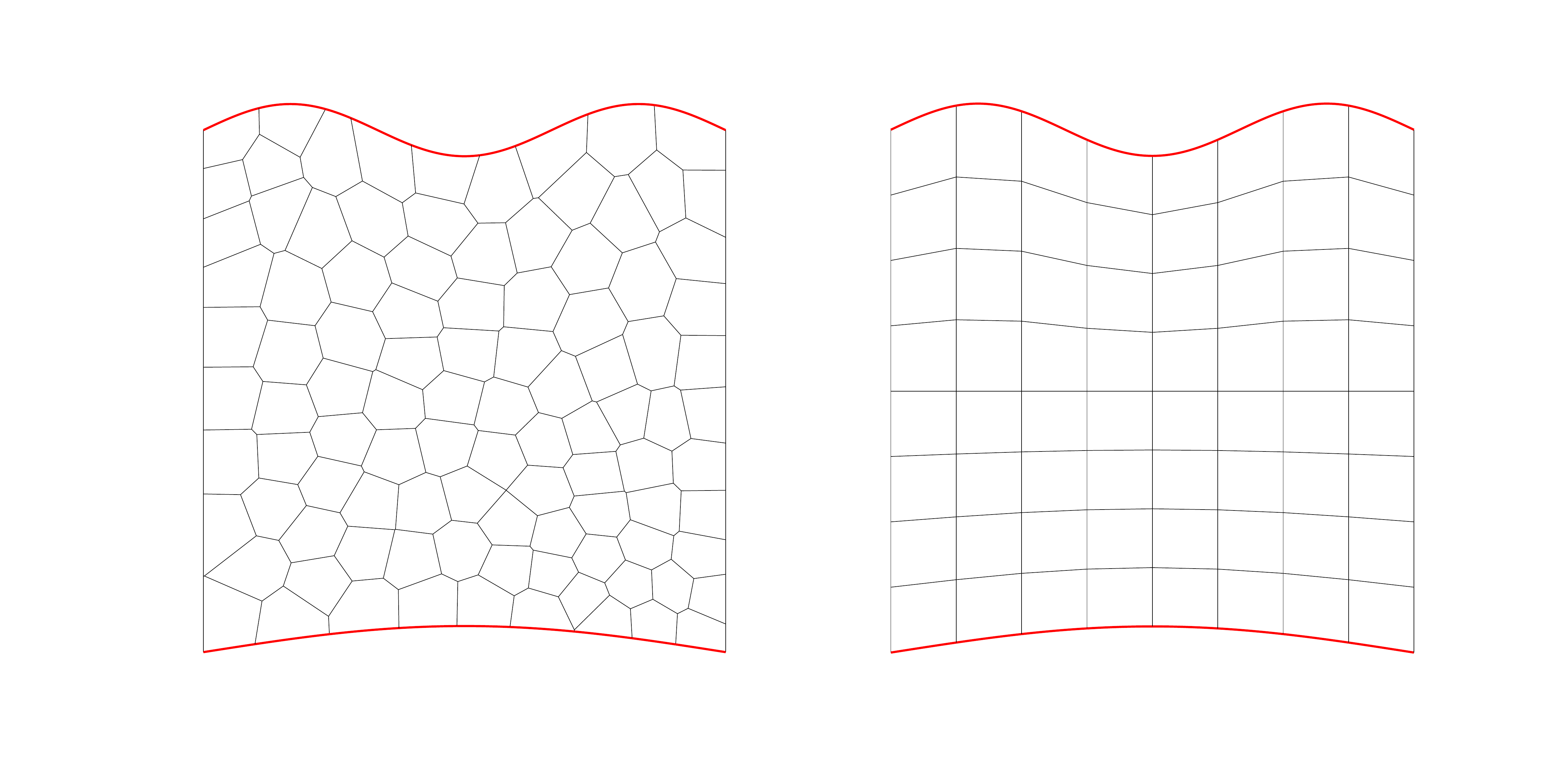} 
\vspace*{-1cm}
\caption{Test \ref{test1}. Example of the adopted curved polygonal meshes: Voronoi mesh on the left and uniform mesh on the right.}
\label{fig:test1meshes}
}
\end{figure}
In Figure \ref{fig:test1-curved-voronoi} we show the results obtained for the sequence of Voronoi meshes and for the sequence of uniform meshes.
We notice that the theoretical predictions of Sections \ref{sec:4} are confirmed. 
\begin{figure}[!h]
\center{
\includegraphics[scale=0.3]{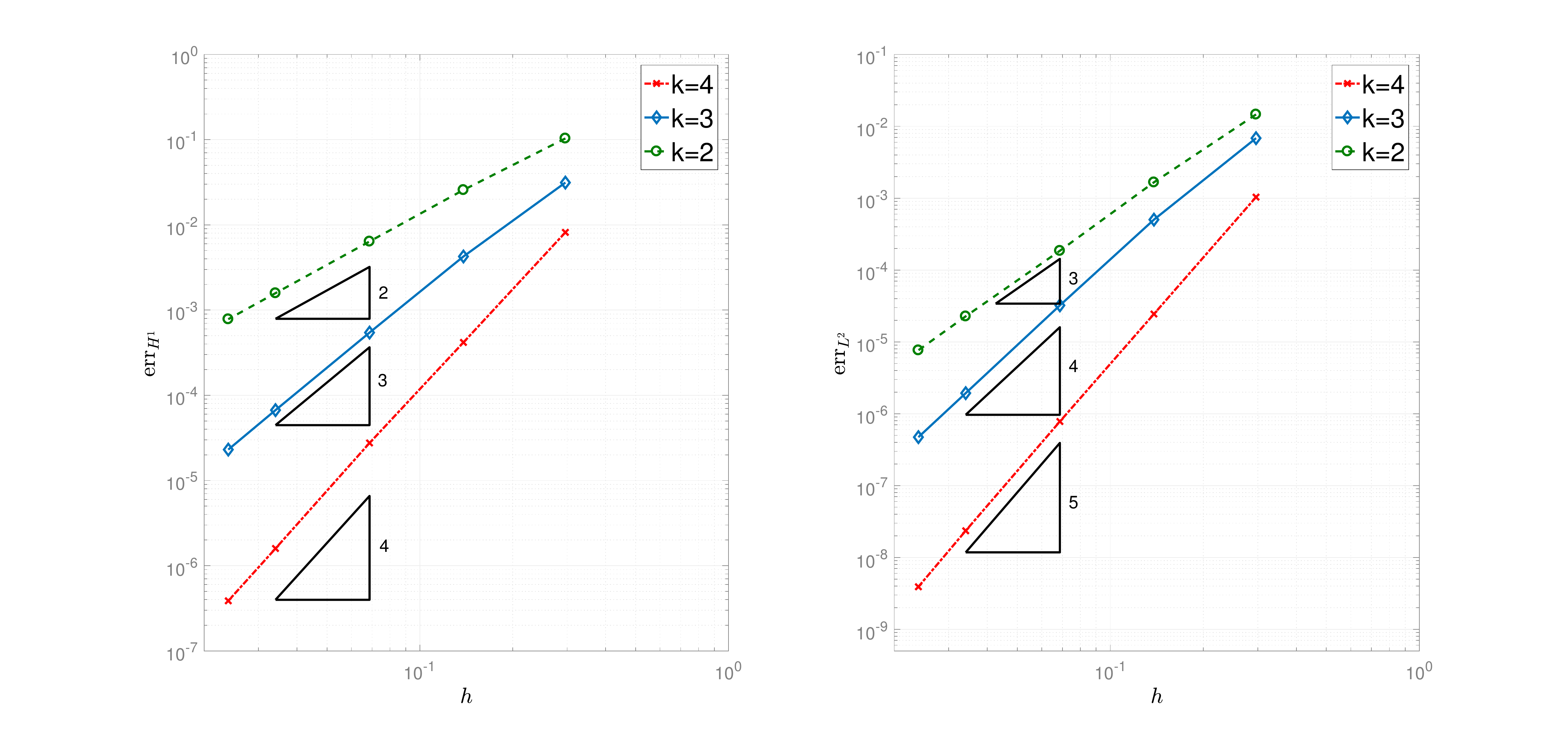} 
\\
\includegraphics[scale=0.3]{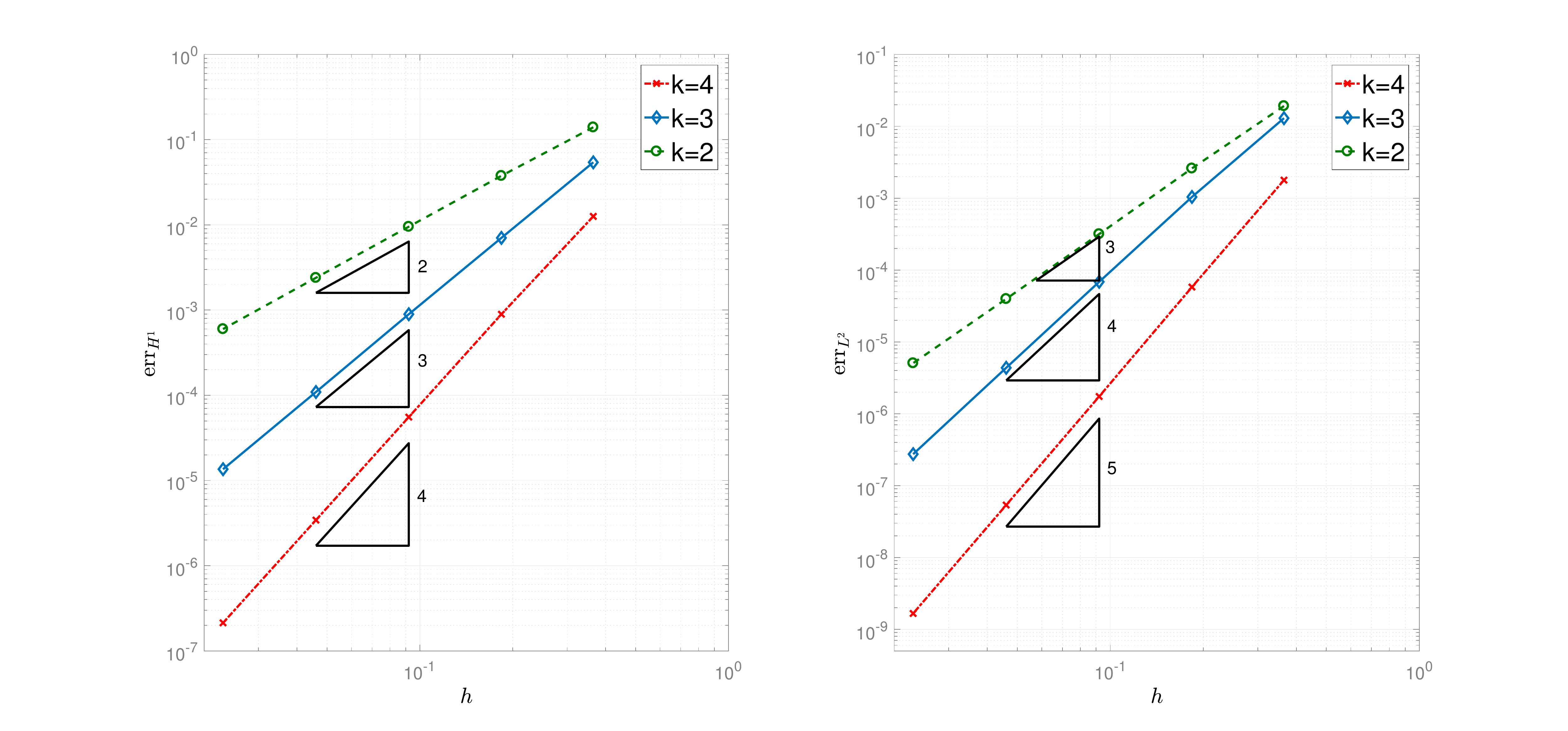} 
\vspace*{-0.5cm}
\caption{Test \ref{test1}. Errors ${\rm err}_{H^1}$ and ${\rm err}_{L^2}$ for the sequence of Voronoi meshes (first row) and for the sequence of uniform meshes (second row).}
\label{fig:test1-curved-voronoi}
}
\end{figure}
Next we approximate the curved domain with a sequence of ``straight polygons'', obtained by approximating the curved boundary by straight segments 
and forcing on these the boundary condition \eqref{eq:test1bc}
(see Figure \ref{fig:test1meshes-straight}).
\begin{figure}[!h]
\center{
\includegraphics[scale=0.25]{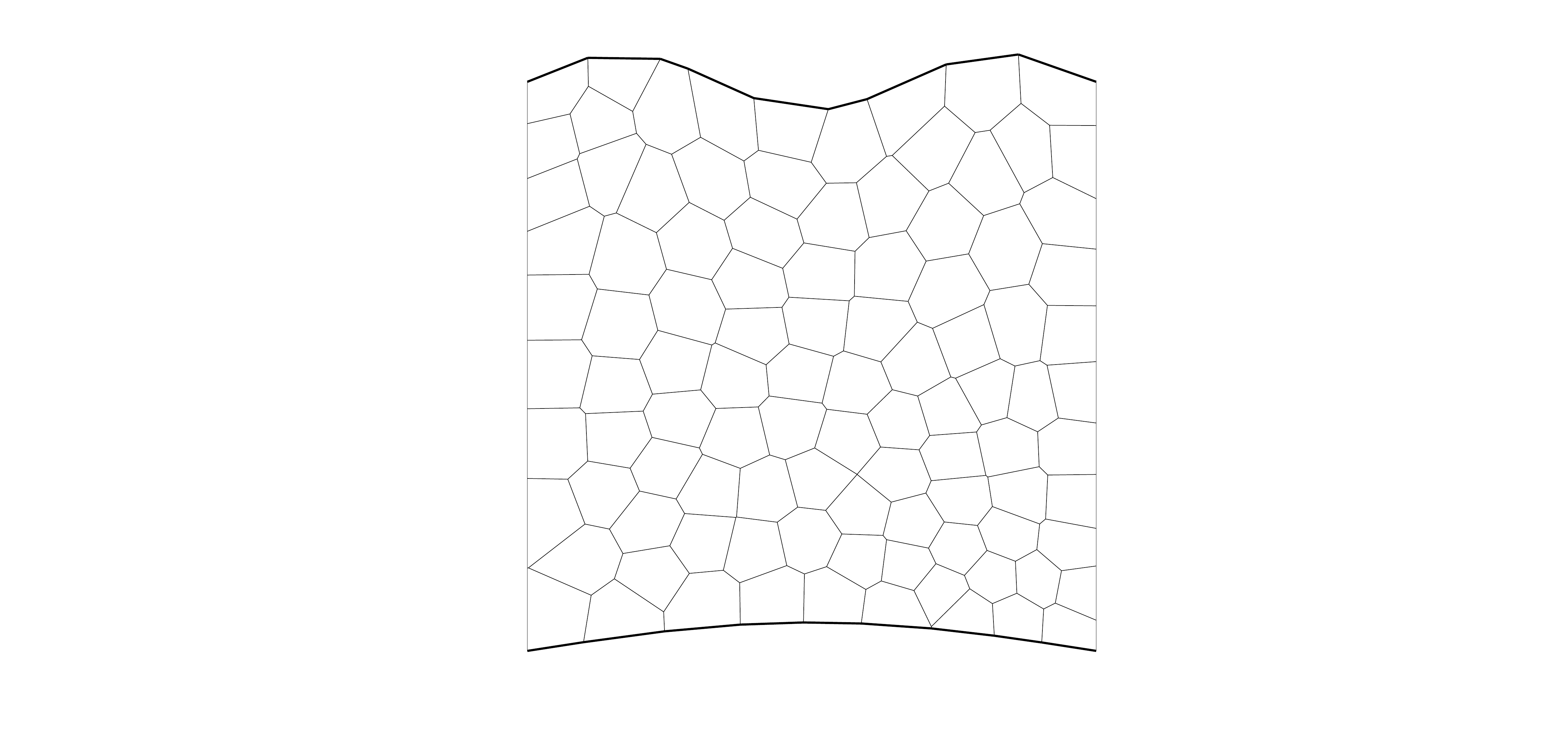} 
\vspace*{-0.5cm}
\caption{Test \ref{test1}. Example of the adopted straight polygonal meshes: Voronoi mesh.}
\label{fig:test1meshes-straight}
}
\end{figure}
In Figure \ref{fig:test1-straight-voronoi} we plot the results for the sequence of Voronoi meshes on the approximated domain,  obtained with the standard VEM on polygons.
\begin{figure}[!h]
\center{
\includegraphics[scale=0.3]{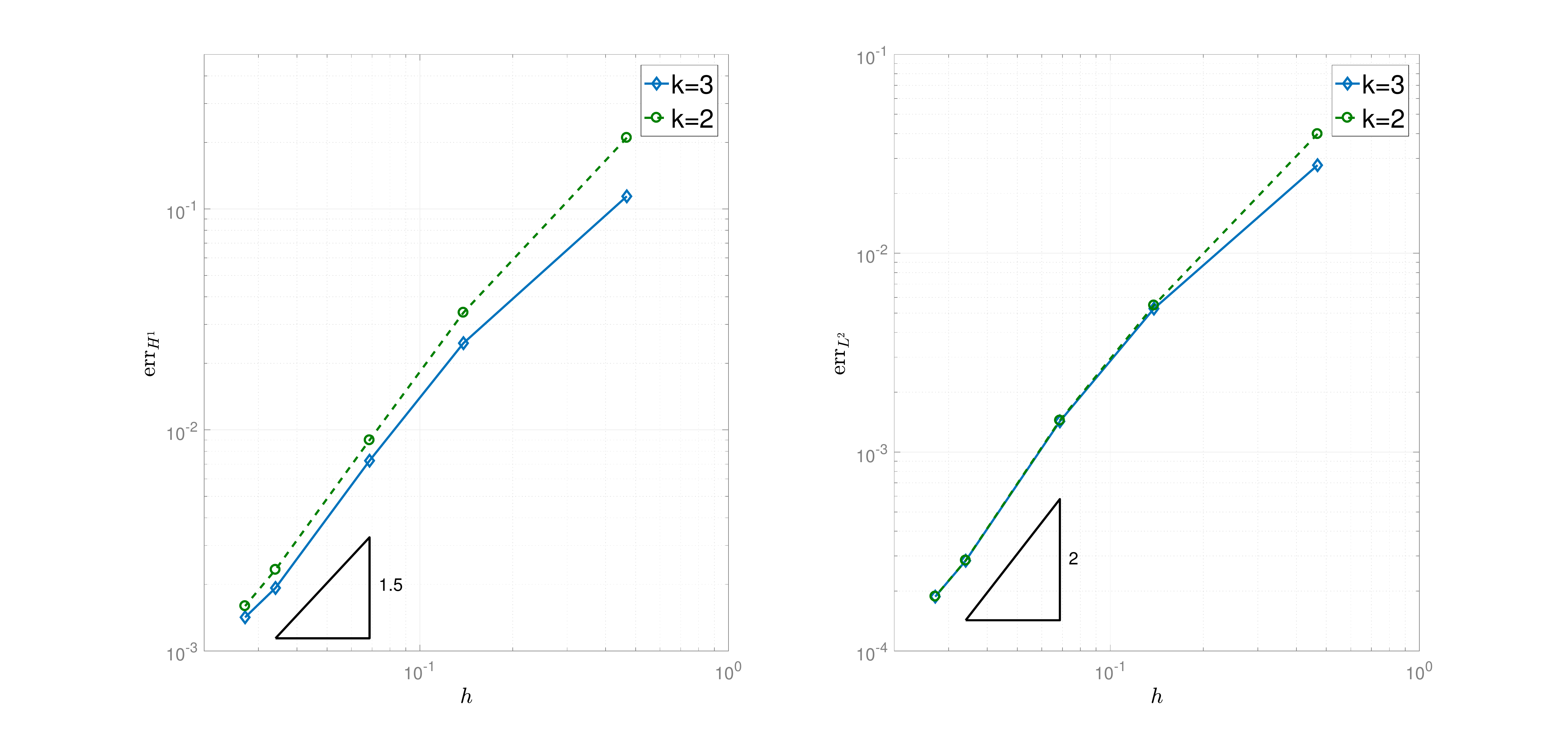} 
\vspace{-0.5cm}
\caption{Test \ref{test1}. Errors ${\rm err}_{H^1}$ and ${\rm err}_{L^2}$ for the sequence of ``straight'' Voronoi meshes.}
\label{fig:test1-straight-voronoi}
}
\end{figure}
As expected, the scheme obtained by the approximation of the domain with straight edge polygons clearly suffers from an evident sub-optimality of the convergence rates.
\end{test}


\begin{test}
\label{test2}
The aim of the present test is to show the good performance of the method also in case of a 
curved domain with fixed curved interface inside the domain.
Let $\Omega := \Omega_1 \cup \Omega_2$, where $\Omega_2$ is the disk centered in $(0,\, 0)$ with radius $1/2$ and $\Omega_1$ is the outer annulus of radius 1 (see Figure \ref{fig:test2domain}).
Let $\Gamma_1$ and $\Gamma_2$ be the circumferences centered in the origin with radius respectively 1 and $1/2$ parametrized by the arc-length parametrizations
\[
\begin{aligned}
\gamma_1 &\colon [0, \, 2 \pi] \to \Gamma_1 
&\qquad \gamma_1(t)& = \left(\cos{t}, \,  \sin{t}\right) ,
\\
\gamma_2 &\colon [0, \, \pi] \to \Gamma_2 
&\qquad \gamma_2(t)& = \frac{1}{2}\left(\cos{2t}, \,  \sin{2t}\right) \,.
\end{aligned}
\]

\begin{figure}[!h]
\center{
%
\begin{overpic}[scale=0.8]{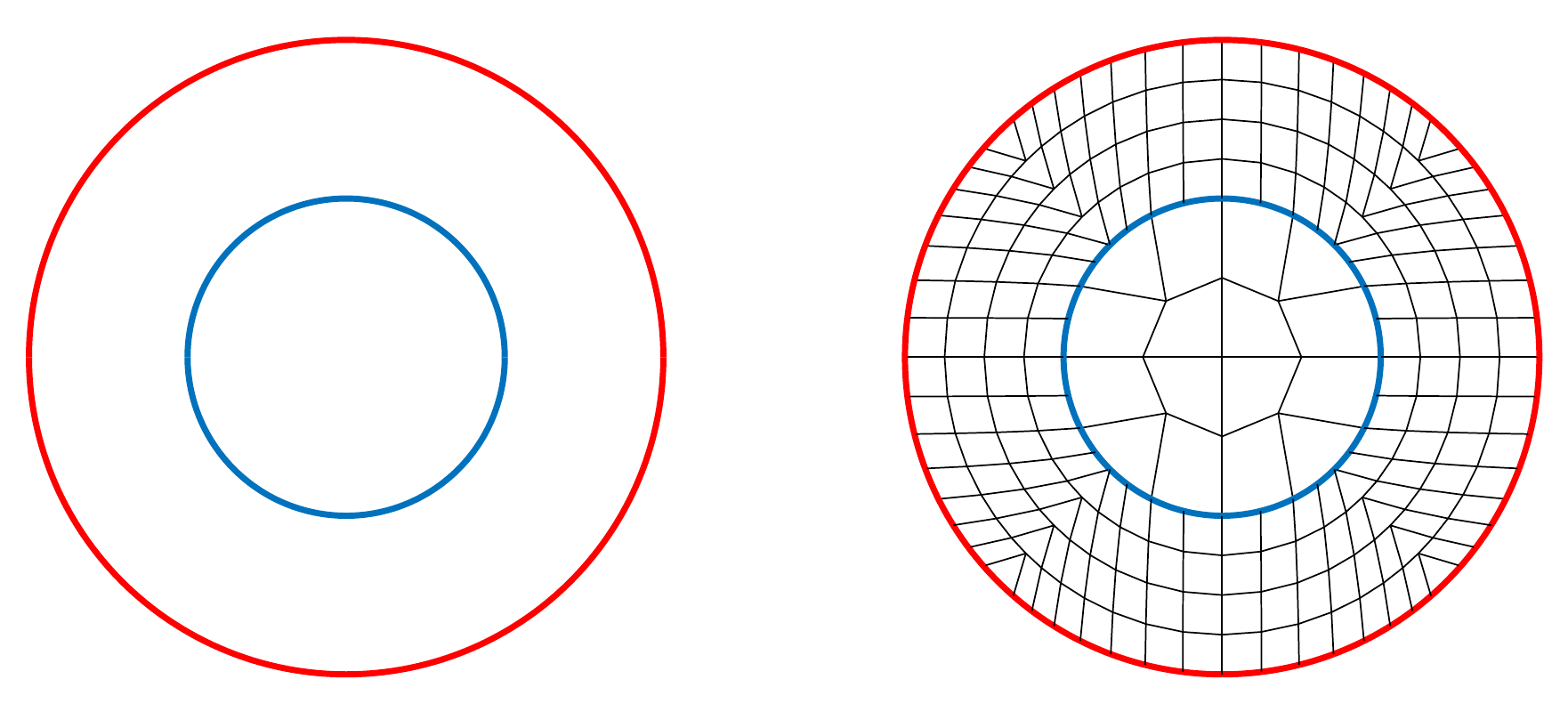}
\put (21,24) {{$\Omega_2$}}
\put (21,9) {{$\Omega_1$}}
\put (30.5,30) {{$\Gamma_2$}}
\put (0.5,15) {{$\Gamma_1$}}
\end{overpic}
\caption{Test \ref{test2}. On the left $\Omega =  \Omega_1 \cup \Omega_2$, curved boundary $\Gamma_1 = \partial \Omega$ and curved interface $\Gamma_2 = \partial \Omega_2$.
On the right example of decomposition of the domain $\Omega$ that matches the curved interface of $\Omega_1$ and $\Omega_2$}
\label{fig:test2domain}
}
\end{figure}
\noindent
We consider the elliptic problem
\[
\left \{
\begin{aligned}
 -{\rm div} \left( \kappa \nabla u \right) &= f &\qquad &\text{in $\Omega$,} \\
 u &= 0 &\qquad &\text{on $\Gamma_1$,} \\
\end{aligned}
\right.
\]
where the viscosity $\kappa$ and the load $f$ are defined by
\[
\left \{
\begin{aligned}
\kappa &= 5 &\qquad &\text{in $\Omega_1$,} \\
\kappa &= 1 &\qquad &\text{on $\Omega_2$,} \\
\end{aligned}
\right.
\qquad
\text{and}
\qquad
\left \{
\begin{aligned}
f &= 1 &\qquad &\text{in $\Omega_1$,} \\
f &= 5 &\qquad &\text{on $\Omega_2$.} \\
\end{aligned}
\right.
\]
Due to the jump in the viscosity and in the loading term, the exact solution 
\[
u_{\rm ex}(x, \, y) = u_{\rm ex}(r) = 
\left\{
\begin{aligned}
&  -  \frac{1}{20} r^2 - \frac{1}{10} \ln(r) +  \frac{1}{20}
\qquad & \text{in $\Omega_1$,} \\
&  -  \frac{5}{4} r^2 + \frac{7}{20} + \frac{\ln(2)}{10}
\qquad & \text{in $\Omega_2$,}\\
\end{aligned}
\right.
\]
is analytical in both subdomains but is not globally regular in $\Omega$ (see Figure \ref{fig:test2solution}).
In order to obtain a method that converges with optimal order, we need to use  decompositions of the domain $\Omega$ which exactly match the curved interface $\Gamma_2$ between the subdomains $\Omega_1$ and $\Omega_2$ (see Figure \ref{fig:test2domain}).
\begin{figure}[!h]
\center{
\hspace*{-1.8cm}
\begin{overpic}[scale=0.35]{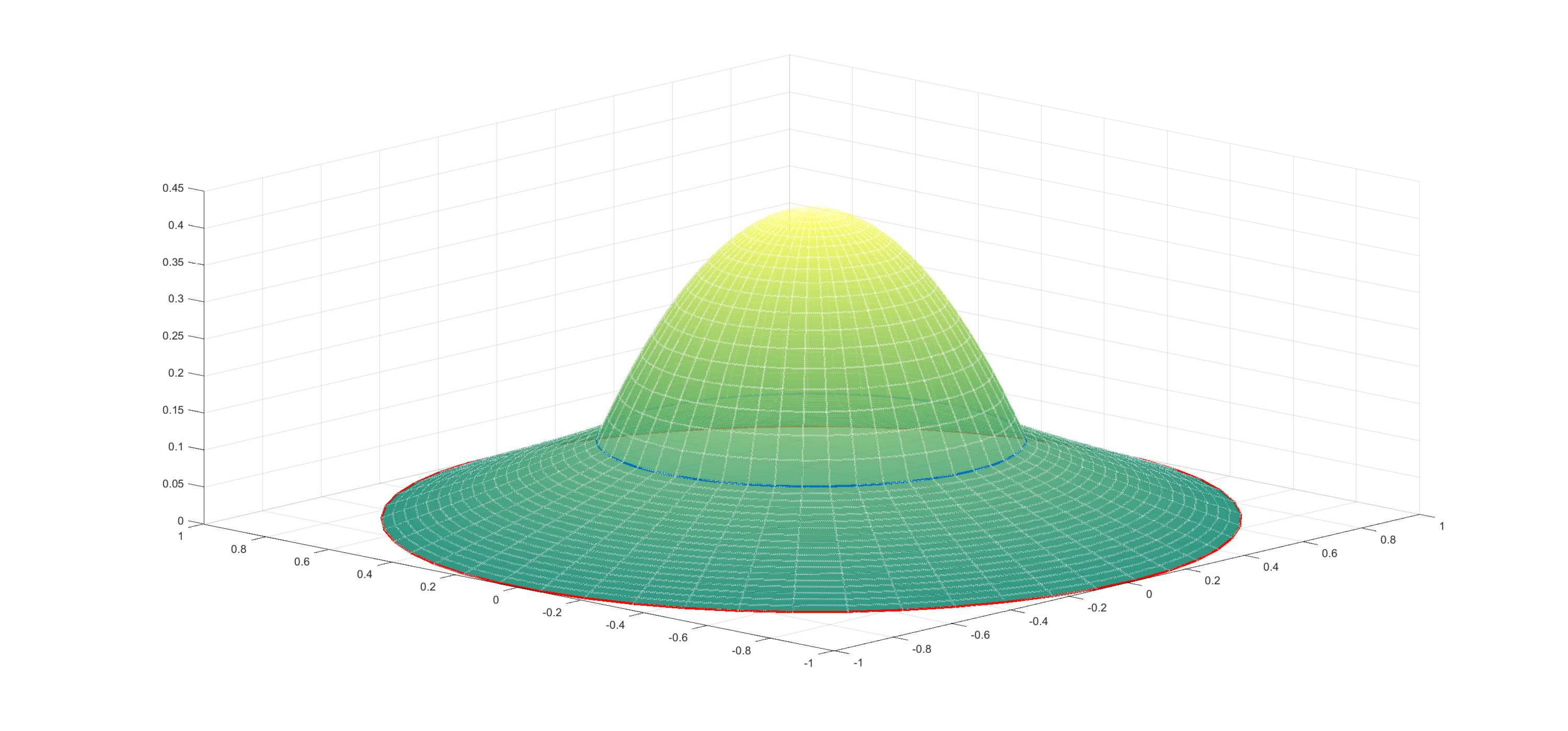}
\put (50,18) {{$\Omega_2$}}
\put (50,12) {{$\Omega_1$}}
\put (36,18) {{$\Gamma_2$}}
\put (24,10) {{$\Gamma_1$}}
\end{overpic}
\vspace*{-1.2cm}
\caption{Test \ref{test2}. Plot of the exact solution $u_{\rm ex}$.}
\label{fig:test2solution}
}
\end{figure}
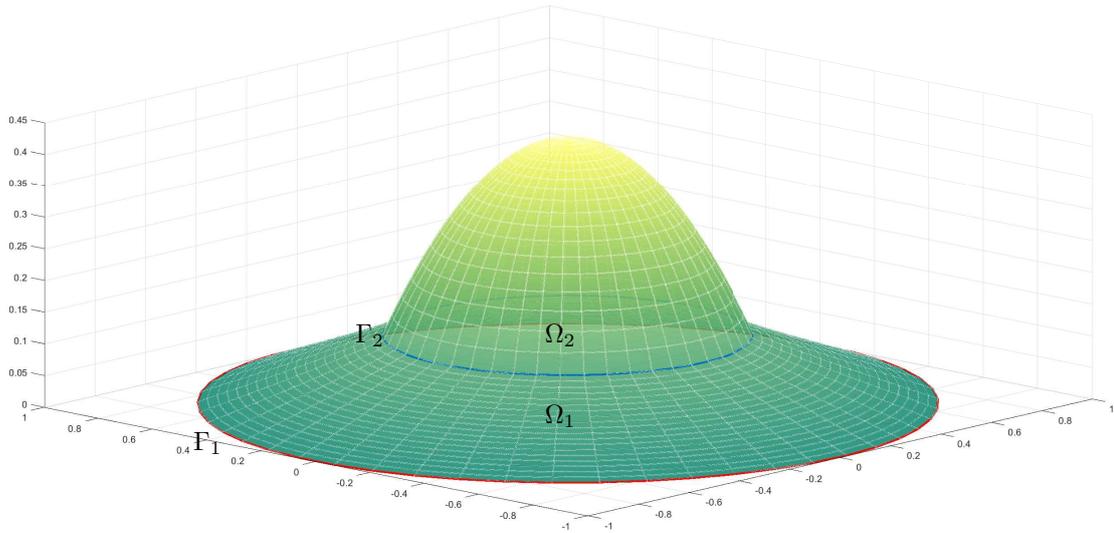

In Figure \ref{fig:test2-curved} we plot the results obtained with the proposed VEM scheme. We observe that the method exhibits appropriate convergence properties, confirming that the proposed virtual element method 
can automatically handle also domains with internal curved interface.
\begin{figure}[!h]
\center{
\includegraphics[scale=0.3]{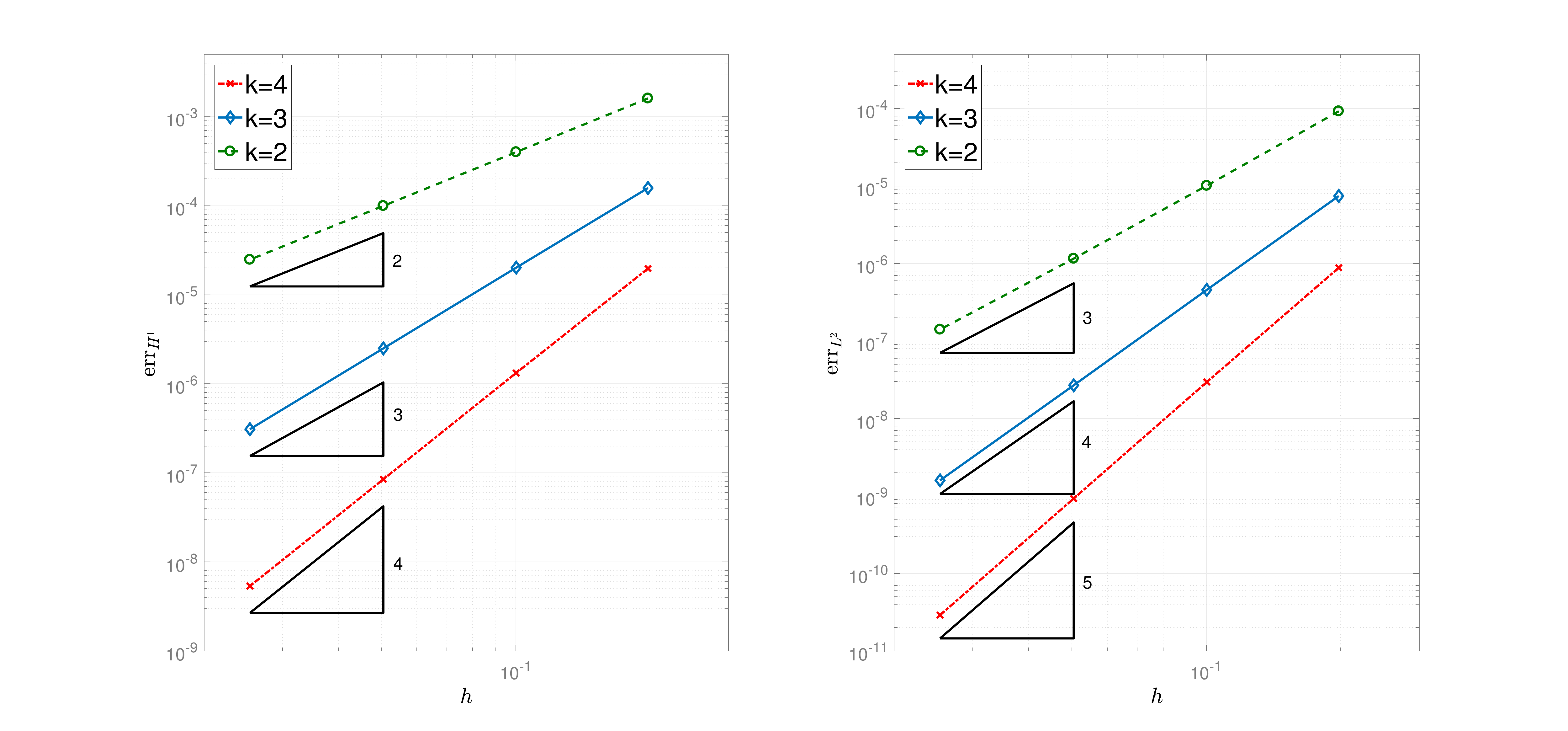} 
\caption{Test \ref{test2}. Errors ${\rm err}_{H^1}$ and ${\rm err}_{L^2}$ for the sequence of meshes matching the interface of $\Omega_1$ and $\Omega_2$.}
\label{fig:test2-curved}
}
\end{figure}

\end{test}

\section{Acknowledgments}
The first and third authors were partially supported by the European Research Council through the H2020
Consolidator Grant (grant no. 681162) CAVE - Challenges and Advancements in Virtual Elements. This support
is gratefully acknowledged.

\addcontentsline{toc}{section}{\refname}
\bibliographystyle{plain}
\bibliography{biblio}

\end{document}